\documentclass{amsart}

\usepackage[dvipsnames]{xcolor}
\usepackage{amssymb, pstricks,pst-plot,pst-text,pst-tree,pst-eps,pst-fill,pst-node,pst-math,dynkin-diagrams,comment,tikz}

\usepackage[all]{xy}

\newcounter{psFigures}     % 0 : sans figures, 1 : avec figures
\newcommand{\psfigure}[1]{
\ifnum \arabic{psFigures} = 0 \fi \ifnum \arabic{psFigures} = 1 #1
\fi }

\usepackage[T1]{fontenc}
\usepackage[latin1]{inputenc}

\newcounter{Version}       
\ifnum \arabic{Version} = 0
\usepackage{showlabels}
\fi

\ifnum \arabic{Version} = 0
\newenvironment{pe}
{\color{orange}\begin{quote}\baselineskip=10pt\noindent
Pierre-Emmanuel: $\blacktriangleright$\ \small\sf}
{\hskip1em\hbox{}\nobreak\hfill$\blacktriangleleft$\par
\end{quote}}
\newenvironment{thomas}
{\color{blue}\begin{quote}\baselineskip=10pt\noindent Thomas:
$\blacktriangleright$\ \small\sf}
{\hskip1em\hbox{}\nobreak\hfill$\blacktriangleleft$\par
\end{quote}}

\fi

\ifnum \arabic{Version} = 1 
\excludecomment{pe}
\excludecomment{thomas} 
\excludecomment{lucas} \fi

\numberwithin{equation}{section}

\newtheorem{theorem}{Theorem}[section]
\newtheorem{lemma}[theorem]{Lemma}
\newtheorem{proposition}[theorem]{Proposition}
\newtheorem{corollary}[theorem]{Corollary}
\theoremstyle{definition}
\newtheorem{definition}[theorem]{Definition}

\newtheorem{remark}[theorem]{Remark}
\newtheorem{notation}[theorem]{Notation}
\newtheorem{example}[theorem]{Example}

\newcommand{\K}{\mathbb{K}}
\newcommand{\Q}{\mathbb{Q}}
\newcommand{\Z}{\mathbb{Z}}

\newcommand{\cH}{{\mathcal{H}}}
\newcommand{\cM}{{\mathcal{M}}}
\newcommand{\cP}{{\mathcal{P}}}

\newcommand{\gl}{{\mathfrak{gl}}}

\newcommand{\fgg}{{\mathfrak{g}}}
\newcommand{\fg}{{\mathfrak{g}}}
\newcommand{\fl}{\mathfrak{l}}
\newcommand{\ft}{\mathfrak{t}}
\newcommand{\fz}{\mathfrak{z}}
\newcommand{\fsl}{\mathfrak{sl}}
\newcommand{\fs}{\mathfrak{s}}

\newcommand{\SL}{\mathrm{SL}}

\newcommand{\lW}[1]{{}^{#1}W}

\newcommand{\eW}{1}

\newcommand{\leqD}{\leq_\mathcal{D}}

\def\cB{{\mathcal{B}}}
\def\cP{{\mathcal{P}}}

\newcommand{\rootfour}[4]{#1 \alpha_1 + #2 \alpha_2 + #3 \alpha_3 + #4 \alpha_4}

\newcommand\leqdot{\mathrel{\ooalign{$\leq$\cr
  \hidewidth\raise0.225ex\hbox{$\cdot\mkern0.5mu$}\cr}}}

\newcommand{\cO}{{\mathcal O}}

\newcommand{\scal}[1]{\langle #1 \rangle}

\begin{document}

\title[Structure of certain spherical quotients]{Parametrization, structure and Bruhat order of certain spherical quotients}

\author[P.-E. Chaput]{Pierre-Emmanuel Chaput}
\address{Universit\'e de Lorraine, CNRS, Institut \'Elie Cartan de Lorraine, UMR 7502, Vandoeu\-vre-l\`es-Nancy, F-54506, France}
\email{pierre-emmanuel.chaput@univ-lorraine.fr}
\author[L. Fresse]{Lucas Fresse}
\address{Universit\'e de Lorraine, CNRS, Institut \'Elie Cartan de Lorraine, UMR 7502, Vandoeu\-vre-l\`es-Nancy, F-54506, France}
\email{lucas.fresse@univ-lorraine.fr}
\author[T. Gobet]{Thomas Gobet}
\address{Institut Denis Poisson, CNRS UMR 7350, Faculté des Sciences et Techniques, Université de Tours, Parc de Grandmont,
37200 TOURS, France} \email{thomas.gobet@lmpt.univ-tours.fr}
\thanks{The first two authors are supported in part by the ANR project GeoLie ANR-15-CE40-0012. The last author was partially supported by the same project and by the ARC project DP170101579.}

\date{\today}

\begin{abstract}
Let $G$ be a reductive algebraic group and let $Z$ be the stabilizer
of a nilpotent element $e$ of the Lie algebra of $G$. We consider
the action of $Z$ on the flag variety of $G$, and we focus on the
case where this action has a finite number of orbits (i.e., $Z$ is a
spherical subgroup). This holds for instance if $e$ has height $2$.
In this case we give a parametrization of the $Z$-orbits and we show
that each $Z$-orbit has a structure of algebraic affine bundle. In
particular, in type $A$, we deduce that each orbit has a natural
cell decomposition. In the aim to study the (strong) Bruhat order of
the orbits, we define an abstract partial order on certain quotients
associated to a Coxeter system. In type $A$, we show that the Bruhat
order of the $Z$-orbits can be described in this way.
\end{abstract}

\maketitle

\part*{Introduction}

Let $G$ be a connected reductive algebraic group over $\K$, where
$\K$ denotes an algebraically closed field of characteristic zero.
Let $B\subset G$ be a Borel subgroup. A closed subgroup $H\subset G$
is said to be spherical if the homogeneous space $G/H$ has a finite
number of $B$-orbits, equivalently if the flag variety
$\mathcal{B}:=G/B$ has a finite number of $H$-orbits.

Let $\mathfrak{g}$ be the Lie algebra of $G$. Let $e\in\mathfrak{g}$
be a nilpotent element. The following subgroups can be attached to
$e$: the stabilizer
\[Z_G(e):=\{g\in G:g\cdot e=e\}\]
and the normalizer
\[N_G(e):=\{g\in G:g\cdot(\K e)=\K e\}.\]
Our concern is the action of $B$ on the nilpotent orbit $G\cdot
e=G/Z_G(e)$ or equivalently the action of $Z_G(e)$ on $G/B$. We
focus on the case where this action comprises a finite number of
orbits, which means that the subgroup $Z_G(e)$ is spherical. In this
situation, our motivation is to understand
\begin{enumerate}
\item the parametrization and the structure of the $Z_G(e)$-orbits on $\mathcal{B}$;
\item the inclusion relations between the orbit closures.
\end{enumerate}
Beyond the fact that the $Z_G(e)$-orbits give an interesting
partition of the whole flag variety, the action of $Z_G(e)$
restricts to certain subvarieties of $\mathcal{B}$ (Springer
varieties, Hessenberg varieties) which arise in geometric
representation theory. Information on the structure and the topology
of the orbits may have applications in this direction.

The nilpotent elements $e$ such that $Z_G(e)$ is spherical are
classified by Panyushev \cite{Panyushev1,Panyushev2} as those of
height at most $3$. The problem of classifying the $Z_G(e)$-orbits
of $G/B$ has been considered in certain (mostly, classical) cases in
\cite{BP,Boos-et-al,BR}. The problem of describing the inclusion
relations between the orbit closures has been addressed in
\cite{BP,BR} in certain cases. To the best of our knowledge, there
is no general approach to these problems, although recently
\cite{gandini} associate to any nilpotent element of height $2$ an
involution in the affine Weyl group, and show that the orbit
closures are described by restricting the Bruhat order on the affine
Weyl group.

The paper is divided into three parts. In Part \ref{part1}, we
review the general background on nilpotent elements and nilpotent
orbits, including Panyushev's classification of spherical nilpotent
orbits. Then we focus on the structure of the group $Z_G(e)$ and
point out the following facts. In general, $Z_G(e)$ has a ``Levi
decomposition'' of the form
\[Z_G(e)=L_Z\ltimes U_Z\]
with reductive part $L_Z\subset L$ and unipotent part $U_Z\subset
U$, respectively contained in the Levi subgroup and the unipotent
radical of a suitable parabolic subgroup $P=L\ltimes U$ attached to
$e$. In the case where $Z_G(e)$ is spherical, it turns out that the
reductive part $L_Z$ coincides with the subgroup of fixed points of
an involution $\sigma\in \mathrm{Aut}(L)$ (possibly up to connected
components). This follows from Panyushev's classification but we
provide a direct argument, defining the involution as the nontrivial
element of the Weyl group of an $\mathrm{SL}_2$-subgroup associated
with the nilpotent element $e$ (see Proposition
\ref{P.symmetric-space}). Finally, in the more particular situation
where $e$ is a nilpotent element of height 2, the unipotent part
$U_Z$ coincides with the full unipotent radical $U$. Hence $Z_G(e)$
is obtained through parabolic induction from a symmetric subgroup of
$L$ in this case.

One of the main ingredients used for Proposition
\ref{P.symmetric-space} is a classification of spherical nilpotent
orbits in terms of sets of so-called rationally orthogonal roots
(see Proposition \ref{P.orbits}), which extends Panyushev's
classification in terms of orthogonal simple roots. The notion of
rationally orthogonal roots differs from those of orthogonal or
strongly orthogonal sets of roots, and this difference is thoroughly
discussed in Section \ref{SI.3}.

Another aspect considered in Part \ref{part1} is the comparison
between the $Z_G(e)$-orbits and $N_G(e)$-orbits of $\mathcal{B}$. We
show that under certain circumstances (including the case where
$Z_G(e)$ is spherical), both sets of orbits actually coincide
(Proposition \ref{PI.3.1-bis}). We however point out an example
which shows that the subgroup $N_G(e)$ may be spherical whereas
$Z_G(e)$ is not. We believe that the comparison of these two sets of
orbits may be a problem of independent interest.

In Part \ref{part2}, we focus on a spherical subgroup of the form
\[H=M\ltimes U\]
where $P=L\ltimes U$ is the Levi decomposition of a parabolic
subgroup and $M\subset L$ is a spherical subgroup. In particular,
the role of $H$ can be played by the stabilizer $Z_G(e)$ of a
nilpotent element of height 2. For $H$ as above, we show that the
$H$-orbits of $\mathcal{B}$ are naturally parametrized by the set
\[W^P\times (\mathcal{B}_L/M)\]
where $W^P$ is the Weyl group parabolic quotient associated to $P$
(i.e., the set of representatives of minimal length of the quotient
$W/W_P$) and $\mathcal{B}_L$ is the flag variety of $L$. Moreover,
in our main result (Theorem \ref{T4.1}) we prove that each orbit has
a structure of algebraic affine bundle over an $M$-orbit of
$\mathcal{B}_L$. In type $A$, we deduce that each $Z_G(e)$-orbit has
a natural cell decomposition (Example \ref{E7.3}).

In Part \ref{part3}, we focus on the (strong) Bruhat order of the
$Z_G(e)$-orbits. We introduce a combinatorial order which reflects
the geometric situation described above. Specifically, given a
Coxeter system $(W,S)$, we consider a parabolic subgroup $W_L\subset
W$ equipped with an involution $\theta:W_L\to W_L$. Then, we
introduce a partial order on the quotient $W/W_L^\theta$, where
$W_L^\theta$ stands for the subgroup of fixed points of $\theta$. We
mostly address the situation where $W_L^\theta$ is a diagonal
subgroup of $W_L$. We investigate certain properties of this order
(minimal representatives, cover relations).

In type $A_{n-1}$, for a nilpotent element $e$ of height $2$, the
$Z_G(e)$-orbits of the flag variety $\mathcal{B}$ are parametrized
by a quotient of the above-mentioned form, namely
$\mathfrak{S}_n/(\Delta\mathfrak{S}_r\times\mathfrak{S}_{n-2r})$,
where $\Delta\mathfrak{S}_r$ stands for the diagonal embedding of
$\mathfrak{S}_r$ into $\mathfrak{S}_r\times\mathfrak{S}_r$. Then,
translating the results of \cite{BP,BR} into our framework, we show
that our combinatorial order coincides with the Bruhat order of the
$Z_G(e)$-orbits.

\bigskip
\noindent {\bf Acknowledgement:}{ we thank Jacopo Gandini for valuable comments which helped improving this paper.}

\setcounter{tocdepth}{1} \tableofcontents

\newpage

\part{Structure of the isotropy group of a spherical nilpotent orbit}

\label{part1}

Throughout this part of the paper, we fix a nilpotent element
$e\in\mathfrak{g}$ and our aim is to describe the structure of its
stabilizer
\[Z:=Z_G(e)=\{g\in G:\mathrm{Ad}(g)e=e\}.\]
We are mostly concerned with the case where the corresponding orbit
$\mathcal{O}_e:=G\cdot e$ is a spherical variety, i.e., it consists
of a finite number of $B$-orbits. Equivalently this means that $Z$
is a spherical subgroup, i.e., the flag variety $\mathcal{B}$ has a
finite number of $Z$-orbits.

In Section \ref{SI.1}, we introduce the basic ingredients which are
useful for describing the structure of $Z$: namely, we recall the
notions of standard triple, cocharacter $\tau$ and parabolic
subgroup $P$ associated to $e$. In particular we recall that the
stabilizer has a ``Levi decomposition''
\[Z=L_Z\ltimes U_Z.\]
In general the subgroup $L_Z$ is not connected. Note that in general
the group $Z$ may not be connected even in the case where the orbit
$\mathcal{O}_e$ is spherical.

The nilpotent orbits which are spherical are classified in
\cite{Panyushev1,Panyushev2} and this classification is recalled in
Proposition \ref{P4}. In particular, one characterization is that
every spherical nilpotent orbit has a representative obtained as sum
of root vectors corresponding to a set of pairwise orthogonal simple
roots; see Proposition \ref{P4}\,{\rm (iii)}. In Proposition
\ref{P.orbits}, we provide a more general classification in terms of
sets of so-called rationally orthogonal (not necessarily simple)
roots, which is the form that we will need in Proposition
\ref{P.symmetric-space} to show that a spherical nilpotent orbit
defines a symmetric space. In Proposition
\ref{P.orthogonal.spherical}, we give a somewhat more precise result
where we characterize the sets of orthogonal roots (not necessarily
simple nor rationally orthogonal) for which a sum of root vectors
belongs to a spherical nilpotent orbit. In Example
\ref{exam:nil-elements}, we also point out that in various cases,
the sets of orthogonal roots corresponding to spherical nilpotent
orbits can be obtained by chain cascade of roots.

In the case where the orbit $\mathcal{O}_e$ is spherical, we obtain
the following description of the group $Z$. In Section \ref{SI.4} we
show that the ``Levi subgroup'' $L_Z$ of $Z$ is a symmetric subgroup
of a Levi subgroup $L\subset P$ (possibly up to certain connected
components). This fact is already known from \cite[Proposition
3.3]{Panyushev1} (at the level of the Lie algebras), but our proof
is somewhat different.\ In particular we give an explicit
construction of an involution $\sigma\in\mathrm{Aut}(L)$ such that
$L_Z$ coincides with the subgroup $L^\sigma$ of fixed points of
$\sigma$ (up to connected components of $L^\sigma$).

In Section \ref{SI.5}, considering the special case where $e$ is a
nilpotent element of height 2, we point out that $Z$ contains the
unipotent radical of $P$. Thus $Z$ is obtained by parabolic
induction from a symmetric subgroup of a Levi factor of $P$.

We believe that Section \ref{SI.2} is of independent interest. In
that section, we compare the $Z$-orbits on $\mathcal{B}$ with the
orbits of the normalizer $N:=N_G(e)=\{g\in G:\mathrm{Ad}(g)e\in\K
e\}$. In particular we show that in the case where $\mathcal{O}_e$
is spherical, both sets of orbits coincide, whereas this is not the
case in general.

\section{Parabolic subgroup associated to a nilpotent element}

\label{SI.1}

By the Jacobson--Morozov lemma, every nilpotent element
$e\in\mathfrak{g}$ is member of a {\it standard triple}, i.e., there
exist $h,f\in\mathfrak{g}$ such that
\[[h,e]=2e,\quad [h,f]=-2f,\quad [e,f]=h,\]
so that $\mathrm{Span}\{e,h,f\}\subset\mathfrak{g}$ is a subalgebra
isomorphic to $\mathfrak{sl}_2(\K)$.

The semisimple element $h$ gives rise to a $\mathbb{Z}$-grading
\[\mathfrak{g}=\bigoplus_{i\in\mathbb{Z}}\mathfrak{g}(i)\quad\mbox{where}\quad \mathfrak{g}(i)=\{x\in\mathfrak{g}:[h,x]=ix\}.\]
The nonnegative part of the grading $\mathfrak{p}:=\bigoplus_{i\geq
0}\mathfrak{g}(i)$ is a parabolic subalgebra, the zero part of the
grading $\mathfrak{l}:=\mathfrak{g}(0)=\mathfrak{z}_\mathfrak{g}(h)$
is a Levi subalgebra of $\mathfrak{p}$, and the positive part of the
grading $\mathfrak{u}:=\bigoplus_{i>0}\mathfrak{g}(i)$ is the
nilpotent radical of $\mathfrak{p}$. Correspondingly the grading
yields a parabolic subgroup $P$ and a Levi decomposition
\[P=L\ltimes U\]
such that $\mathfrak{p}=\mathrm{Lie}(P)$,
$\mathfrak{l}=\mathrm{Lie}(L)$, and $\mathfrak{u}=\mathrm{Lie}(U)$.

By the representation theory of $\mathfrak{sl}_2(\K)$, we have the
inclusion
\[\mathfrak{z}_\mathfrak{g}(e)=\mathrm{Lie}(Z)\subset \bigoplus_{i\geq 0}\mathfrak{g}(i)=\mathfrak{p}\]
and the dimension formula
\[\dim\mathfrak{z}_\mathfrak{g}(e)=\dim \mathfrak{g}(0)+\dim \mathfrak{g}(1).\]

Moreover, there is a (unique) cocharacter $\tau:\K^*\to G$  such
that $\tau'(1)=h$. This implies that the parabolic subgroup
$P\subset G$ and its Levi decomposition $P=L\ltimes U$ can also be
characterized as follows:
\begin{eqnarray*}
 & \displaystyle P=\{g\in G:\lim_{t\to 0}\tau(t)g\tau(t)^{-1}\mbox{ exists}\}, \\
 & \displaystyle L=\{g\in G:\forall t\in\K^*,\ \tau(t)g\tau(t)^{-1}=g\}=Z_G(\tau), \\
 & \displaystyle U=\{g\in G:\lim_{t\to 0}\tau(t)g\tau(t)^{-1}=1_G\}.
\end{eqnarray*}

\begin{proposition}
\label{P-Levi-Z}
\begin{itemize}
\item[\rm (a)] We have $Z\subset P$. Letting $U_Z:=U\cap Z$ and $L_Z:=L\cap Z$, we have
\[Z=L_Z\ltimes U_Z.\]
Moreover, the subgroup $U_Z$ is connected.
\item[\rm (b)] Considering the connected subgroup $Z^0\subset Z$, we have $L\cap (Z^0)=(L_Z)^0=:L_Z^0$,
$U\cap (Z^0)=U_Z$, and
\[Z^0=L_Z^0\ltimes U_Z.\]
\item[\rm (c)] If $S$ is a maximal torus of $L_Z^0$, then $S$ is a maximal torus of $Z$.
\end{itemize}
\end{proposition}

\begin{proof}
{\rm (a)} Let $g\in Z$. Then $(e,g\cdot h,g\cdot f)$ is a standard
triple which also contains the element $e$. By Kostant's theorem
\cite[Theorem 3.4.10]{Collingwood-McGovern}, there is an element
$u\in U\cap Z$ such that $g\cdot h=u\cdot h$ and $g\cdot f=u\cdot
f$. In particular $u^{-1}g\cdot h=h$, which means that
$u^{-1}g=:\ell\in L\cap Z$. Whence $g=u\ell\in P$. This argument
shows in fact that the inclusion $Z\subset L_ZU_Z$ holds. Since
$L_Z\cap U_Z=\{1_G\}$, we conclude that $Z=L_Z\ltimes U_Z$.

Letting $u\in U_Z$, we have for all $t\in\K^*$
\[\mathrm{Ad}(\tau(t)u\tau(t)^{-1})e=t^{-2}\mathrm{Ad}(\tau(t)u)e=t^{-2}\mathrm{Ad}(\tau(t))e=e\]
hence $\{\tau(t)u\tau(t)^{-1}:t\in\K^*\}\subset U_Z$.\ Since
$\lim_{t\to 0}\tau(t)u\tau(t)^{-1}=1_G$, we conclude that $u\in
(U_Z)^0$.\ This shows that $U_Z$ is connected.

{\rm (b)} Since $U_Z$ is connected, we have $U_Z\subset Z^0$, hence
$U_Z=U\cap(Z^0)$. By {\rm (a)}, we get the equality $Z^0=(L\cap
(Z^0))\ltimes U_Z$, which also implies that $L\cap (Z^0)$ is
connected, i.e., $L\cap (Z^0)=(L_Z)^0$.

{\rm (c)} Let $S$ be a maximal torus of $L_Z^0$ and let $T$ be a
maximal torus of $Z$ such that $S\subset T$. Thus $T\subset Z^0$. In
view of {\rm (b)}, there is a surjective morphism of algebraic
groups $\pi:Z^0\to L_Z^0=Z^0/U_Z$. Since $S$ is a maximal torus of
$L_Z^0$, we deduce that $\pi(T)=S$. On the other hand, since every
element of $U_Z$ is unipotent, $T\cap U_Z$ must be trivial.
Therefore, the equality $S=T$ must hold.
\end{proof}

\section{Relation between normalizer and stabilizer}

\label{SI.2}

The normalizer of $e\in\mathfrak{g}$ is the subgroup
\[N:=N_G(e)=\{g\in G:\mathrm{Ad}(g)e\in\K e\}.\]
We have the following relation between $Z$ and $N$; here $\tau$ is
the cocharacter associated to $e$ as in Section \ref{SI.1}.

\begin{proposition}
\label{PI.3.1} $Z$ is a normal subgroup of $N$. We have
$N=Z\{\tau(t)\}_{t\in\K^*}$ and $Z\cap\{\tau(t)\}_{t\in\K^*}$
contains (at most) two elements.
\end{proposition}

\begin{proof}
The first claim is clear. For $g\in N$, we have
$\mathrm{Ad}(g)e=t^2e$ for some $t\in\K^*$, hence
$\mathrm{Ad}(g^{-1}\tau(t))e=e$, i.e., $g\in Z\tau(t)$. Finally, for
$t\in\K^*$, the equality $\mathrm{Ad}(\tau(t))e=t^2e$ implies that
$\tau(t)$ belongs to $Z$ if and only if $t\in\{-1,1\}$.
\end{proof}

\begin{remark}
If the nilpotent element $e$ is even (that is, if the grading of
Section \ref{SI.1} satisfies $\mathfrak{g}(i)=0$ for all odd $i$),
then $\tau(-1)=\tau(1)=1_G$. In this case, we have
$Z\cap\{\tau(t)\}_{t\in\K^*}=\{1_G\}$, so that
$N=Z\rtimes\{\tau(t)\}_{t\in\K^*}$.
\end{remark}

We now compare the actions of $N$ and $Z$ on the flag variety
$\mathcal{B}$. Evidently, each $Z$-orbit is contained in an
$N$-orbit, and this implies that $N$ is a spherical subgroup
whenever $Z$ is a spherical subgroup. There are situations where the
$N$-orbits and the $Z$-orbits of $\mathcal{B}$ actually coincide.

\begin{proposition}
\label{PI.3.1-bis}
\begin{itemize}
\item[\rm (a)]
If every $N$-orbit (or every $Z$-orbit) of $\mathcal{B}$ contains an
element fixed by $\tau$, then the $Z$-orbits of $\mathcal{B}$
coincide with the $N$-orbits. \item[\rm (b)] If $Z$ is a spherical
subgroup of $G$, then the $Z$-orbits of $\mathcal{B}$ coincide with
the $N$-orbits.
\end{itemize}
\end{proposition}

\begin{proof}
{\rm (a)} Assume that every $N$-orbit  of $\mathcal{B}$ contains an
element fixed by $\tau$ (this is more general than assuming that
every $Z$-orbit contains such an element). Then every $N$-orbit
takes the form $N\cdot gB$ with $\tau(t)gB=gB$ for all $t\in\K^*$.
In view of Proposition \ref{PI.3.1}, we get $N\cdot gB=Z\cdot gB$,
which shows that every $N$-orbit is actually a $Z$-orbit.

{\rm (b)} Assume that $Z$ is a spherical subgroup. Let
$\mathcal{N}:=N\cdot gB$ be an $N$-orbit of $\mathcal{B}$; in
particular $\mathcal{N}$ contains finitely many $Z$-orbits. We
consider the map
\[\phi:\K^*\to\mathcal{N},\ t\mapsto \phi(t):=\tau(t)gB.\]
Note that every $Z$-orbit of $\mathcal{N}$ is of the form
$Z\cdot\tau(t)gB$ for some $t\in\K^*$. We have
\begin{eqnarray*}
\dim Z\cdot\tau(t)gB & = & \dim Z-\dim\{h\in Z:h\tau(t)gB=\tau(t)gB\} \\
 & = & \dim Z-\dim \tau(t)^{-1}Z\tau(t)\cap gBg^{-1} \\
 & = & \dim Z-\dim Z\cap gBg^{-1}
\end{eqnarray*}
because $\tau(t)^{-1}Z\tau(t)=Z$.\ This shows that all the
$Z$-orbits contained in $\mathcal{N}$ have the same dimension. In
particular this implies that
\[\mathcal{Z}=\overline{\mathcal{Z}}\cap\mathcal{N}\quad\mbox{for every $Z$-orbit $\mathcal{Z}\subset\mathcal{N}$}.\]
Since the map $\phi$ is algebraic, we obtain
\[\{s\in\K^*:\phi(s)\in \mathcal{Z}\}\mbox{ is closed in $\K^*$}\quad\mbox{for every $Z$-orbit $\mathcal{Z}\subset\mathcal{N}$}.\]
Since $\mathcal{N}$ contains a finite number of $Z$-orbits, there is
a $Z$-orbit $\mathcal{Z}\subset\mathcal{N}$ such that
$\{s\in\K^*:\phi(s)\in\mathcal{Z}\}=\K^*$, which just means that
$\mathcal{Z}=\mathcal{N}$.
\end{proof}

\begin{remark}
{\rm (a)} In general, $Z$-orbits and $N$-orbits of $\mathcal{B}$ do
not coincide. It actually may happen that $N$ is spherical whereas
$Z$ is not, as shown by the following example. Let us consider the
situation where $\mathfrak{g}=\mathfrak{sl}_3(\mathbb{K})$ and
\[e=\left(\begin{array}{ccc}
0 & 1 & 0 \\ 0 & 0 & 1 \\ 0 & 0 & 0
\end{array}\right).\]
In this case
\begin{eqnarray*}
& Z=\left\{ \left(\begin{array}{ccc} a & b & c \\ 0 & a & b \\ 0 & 0
& a
\end{array}\right):a\in\K^*\mbox{ s.t. $a^3=1$};\ b,c\in\K
\right\},\\
& N=\{\tau(t)\}_{t\in\K^*}Z\quad\mbox{with}\quad\tau(t)=
\left(\begin{array}{ccc} t^2 & 0 & 0 \\ 0 & 1 & 0 \\ 0 & 0 & t^{-2}
\end{array}\right).
\end{eqnarray*}
Let $(\varepsilon_1,\varepsilon_2,\varepsilon_3)$ be the standard
basis of $\K^3$. Here the flag variety $\mathcal{B}$ can be viewed
as the set of complete flags of $\K^3$. Such a complete flag
consists of a pair $(F_1,F_2)$ with $\dim F_1=1$, $\dim F_2=2$,
$0\subset F_1\subset F_2\subset \K^3$. Given linearly independent
vectors $v_1,v_2$, we write $F(v_1,v_2):=(\langle v_1\rangle,\langle
v_1,v_2\rangle)$. It is easy to see that the elements
$F(\varepsilon_3,\varepsilon_1+t\varepsilon_2)$, for $t\in\K$,
belong to pairwise distinct $Z$-orbits of $\mathcal{B}$. Thus
$\mathcal{B}$ has infinitely many $Z$-orbits. However, $\mathcal{B}$
has exactly seven $N$-orbits, whose representatives are
\[
F(\varepsilon_1,\varepsilon_2),\ F(\varepsilon_1,\varepsilon_3),\
F(\varepsilon_2,\varepsilon_1),\ F(\varepsilon_2,\varepsilon_3),\
F(\varepsilon_3,\varepsilon_1),\ F(\varepsilon_3,\varepsilon_2),\
F(\varepsilon_3,\varepsilon_1+\varepsilon_2).
\]
Only the last one of these orbits does not contain any element fixed
by $\tau$.

{\rm (b)} In Remark \ref{R-no-tau-fixedpoint} (b) we point out an
example where $Z$ is spherical (thus $Z$-orbits coincide with
$N$-orbits by Proposition \ref{PI.3.1-bis}\,{\rm (b)}) though there
is a $Z$-orbit of $\mathcal{B}$ which contains no point fixed by
$\tau$. Hence the converse of Proposition \ref{PI.3.1-bis}\,{\rm
(a)} is in general not true.
\end{remark}

\section{Nilpotent orbits and orthogonal roots}

\label{SI.3}

Before stating a list of necessary and sufficient conditions for
$Z=Z_G(e)$ to be a spherical subgroup of $G$, we recall some notions
related to nilpotent elements.

Since $e$ is a nilpotent element, its image by the adjoint
representation is a nilpotent endomorphism
$\mathrm{ad}\,e:\mathfrak{g}\to\mathfrak{g}$. Then, the {\it height}
of $e$ is defined as the biggest integer $k\geq 0$ such that
$(\mathrm{ad}\,e)^k\not=0$. Equivalently, $k$ is maximal such that
$\mathfrak{g}(k)\not=0$, for the grading
$\mathfrak{g}=\bigoplus_{i\in\mathbb{Z}}\mathfrak{g}(i)$ of Section
\ref{SI.1}.

If we take a maximal torus $S$ of $Z$, then the Lie algebra
\[\mathfrak{g}_0(e):=\mathrm{Lie}(Z_G(S))=\{x\in\mathfrak{g}:\forall s\in S,\ \mathrm{Ad}(s)x=x\}\]
is a Levi subalgebra of $\mathfrak{g}$ which contains $e$ and which
is minimal for this property. The Lie algebra $\mathfrak{g}_0(e)$
does not essentially depend on the choice of the torus $S$. The type
of the semisimple Lie algebra
$[\mathfrak{g}_0(e),\mathfrak{g}_0(e)]$ is referred to as the type
of the nilpotent orbit $G\cdot e$. This datum arises in the
classification of nilpotent orbits due to Bala and Carter; see
\cite[\S8]{Collingwood-McGovern} for more details.

In the next statement, we also fix a root space decomposition
$\mathfrak{g}=\mathfrak{t}\oplus\bigoplus_{\alpha\in\Phi}\mathfrak{g}_\alpha$
and a system of positive roots $\Phi^+\subset\Phi$. Every nilpotent
element of $\mathfrak{g}$ lies in the (adjoint) $G$-orbit of an
element of the space
$\bigoplus_{\alpha\in\Phi^+}\mathfrak{g}_\alpha$. There is no loss
of generality in assuming that the image of the cocharacter $\tau$
associated to $e$ is contained in the maximal torus $T\subset G$
such that $\mathfrak{t}=\mathrm{Lie}(T)$.

\begin{proposition}[\cite{Panyushev1,Panyushev2}]
\label{P4} The following conditions are equivalent:
\begin{itemize}
\item[\rm (i)] $Z=Z_G(e)$ is a spherical subgroup of $G$;
\item[\rm (ii)] The height of $e$ is at most 3;
\item[\rm (iii)] $e$ belongs to the $G$-orbit of an element obtained as the sum of root vectors corresponding to pairwise orthogonal simple roots;
\item[\rm (iv)] Every simple factor of $[\mathfrak{g}_0(e),\mathfrak{g}_0(e)]$ is of type $A_1$.
\end{itemize}
\end{proposition}

Note that the condition in Proposition \ref{P4}\,{\rm (iii)} (which
corresponds to \cite[Theorem 3.4]{Panyushev2}) yields a kind of
normal form for spherical nilpotent orbits. We actually need a
slightly different version of Proposition \ref{P4}\,{\rm (iii)},
which we give in the next statement.

Recall that two roots $\alpha,\beta\in\Phi$ are said to be strongly
orthogonal if $\alpha+\beta$ and $\alpha-\beta$ do not belong to
$\Phi\cup\{0\}$. In particular, this implies that $\alpha$ and
$\beta$ are orthogonal. We will need a somewhat more restrictive
condition that we now state:

\begin{thomas}
Il y a un peu un melange de notations entre $\beta_i$ et $\theta_i$. Ici on considere des racines $\beta_1, \beta_2, \dots, \beta_k$
mais dans certains enonces on utilise plutot $\theta_i$ (avec aussi la notation $\theta_k$).
Je ne sais pas si c'est un vrai probleme, on peut aussi laisser comme ça dans la mesure ou ça ne prête pas vraiment a confusion, je crois. 
\end{thomas}

\begin{definition}
Let $\beta_1,\ldots,\beta_k \in \Phi$ be linearly independent. We
say that $\beta_1,\ldots,\beta_k$ are {\em rationally orthogonal} if
$$\left (\bigoplus_{i=1}^k \Q\, \beta_i \right ) \cap \Phi = \{ \pm \beta_1, \pm \beta_2, \ldots, \pm \beta_k \} .$$
\end{definition}

A mistake in a previous version of this paper occurring in the
following Proposition has been pointed out to us by J. Gandini, to
whom we are grateful.

\begin{proposition}
\label{P.orbits}
\begin{itemize}
\item[\rm (a)]
Let $\theta_1,\ldots,\theta_r\in\Phi$ be a sequence of rationally
orthogonal roots (not necessarily simple nor positive). For
$i\in\{1,\ldots,r\}$, let $e_{\theta_i} \in \mathfrak{g}_{\theta_i}
\setminus \{0\}$. Then $e=\sum_i e_{\theta_i}$ belongs to a
spherical nilpotent orbit of type~$rA_1$.
\item[\rm (b)]
Conversely, every spherical nilpotent orbit contains an element of
the form $e=\sum_i e_{\theta_i}$ corresponding to a sequence
$\theta_1,\ldots,\theta_r$ of rationally orthogonal roots, where we
may assume in addition that $\sum_i \theta_i^\vee$ is a dominant
coweight.
\end{itemize}
\end{proposition}

\begin{proof}
{\rm (a)} The element $e = \sum_i e_{\theta_i}$ is nilpotent
because, knowing that $\theta_1,\ldots,\theta_r$ are linearly
independent, there can be only finitely many pairs $(\alpha,(n_i))$
where $\alpha$ is $0$ or a root and $(n_i)=(n_1,\ldots,n_r)$ is a
sequence of integers, such that $\alpha + \sum_i n_i \theta_i$ is
$0$ or a root. Note that $h:=\sum_i\theta_i^\vee$ (as an element of
$\mathfrak{t}$) satisfies $[h,e]=2e$, and we can find
$f_{\theta_i}\in\mathfrak{g}_{-\theta_i}$ such that $(e,h,\sum_i
f_{\theta_i})$ is a standard triple. This implies that
$\tau:=\sum_i\theta_i^\vee$ (seen this time as a cocharacter of $T$)
is a cocharacter associated to $e$ in the sense of Section
\ref{SI.1}.

If $\alpha$ is a root, we denote
$\mathfrak{sl}_2(\alpha):=\mathfrak{g}_{-\alpha}\oplus[\mathfrak{g}_{-\alpha},\mathfrak{g}_\alpha]
\oplus\mathfrak{g}_\alpha$, which is a subalgebra of $\mathfrak{g}$
isomorphic to $\mathfrak{sl}_2(\K)$.

Note that $S:=\bigcap_i\ker \theta_i$ is a torus contained in $Z$,
actually it is also contained in the subgroup $L_Z^0=Z_G(\tau)\cap
Z^0$. In view of Proposition \ref{P-Levi-Z}, there is a maximal
torus $S'$ of $Z$ such that $S\subset S'\subset L_Z^0$.

We claim that the centralizer of $S\cdot\tau(\K^*)$ is $T$. Indeed,
for otherwise, there is a root $\alpha$ which is trivial on
$S\cdot\tau(\K^*)$. Since $\alpha$ is in particular trivial on $S$,
it is a linear combination of the roots $\theta_i$ with rational
coefficients. Since the roots $\theta_i$ are rationally orthogonal,
there is an integer $i$ such that $\alpha = \pm \theta_i$. Finally,
since $\alpha$ must also be trivial on $\tau = \theta_1^\vee +
\cdots + \theta_r^\vee$, we get a contradiction. Thus, the
centralizer of $S \cdot \tau(\K^*)$ is $T$. This implies that
$S'\subset T$, hence
$\mathfrak{t}\subset\mathfrak{z}_\mathfrak{g}(S'):=\{x\in\mathfrak{g}:\forall
s\in S',\ \mathrm{Ad}(s)x=x\}$.

Our argument also shows that
\[\mathfrak{z}_\mathfrak{g}(S) = \mathfrak{t} + \bigoplus_{i=1}^r \mathfrak{sl}_2(\theta_i)\quad\mbox{and so}
\quad
[\mathfrak{z}_\mathfrak{g}(S),\mathfrak{z}_\mathfrak{g}(S)]=\bigoplus_{i=1}^r
\mathfrak{sl}_2(\theta_i).\] Since $\mathfrak{t}+\K
e\subset\mathfrak{z}_\mathfrak{g}(S')\subset\mathfrak{z}_\mathfrak{g}(S)$,
we get $\mathfrak{z}_\mathfrak{g}(S')=\mathfrak{z}_\mathfrak{g}(S)$,
i.e., $\mathfrak{g}_0(e)=\mathfrak{z}_\mathfrak{g}(S)$ with the
notation of Proposition \ref{P4}. Hence
$[\mathfrak{g}_0(e),\mathfrak{g}_0(e)]$ is of type $rA_1$ and, since
the condition in Proposition \ref{P4}\,{\rm (iv)} is satisfied, the
orbit $G\cdot e$ is spherical.

{\rm (b)} Conversely, let $e$ be a nilpotent element and $S$ a
maximal torus of $Z$ contained in $L_Z^0$. Up to the action of an
element of $G$, we may assume that $S \cdot \tau(\K^*) \subset T$.
Assuming that $e$ is spherical, by Proposition \ref{P4}(iii), we may
write $e=\sum_i e_{\theta_i}$ with $\theta_i$ a set of rationally
orthogonal roots.

If the sum $\sum_i \theta_i^\vee$ is not dominant, then there is a
simple root $\alpha$ such that $s_{\alpha}(\sum_i \theta_i^\vee) >
\sum_i \theta_i^\vee$. Then,
$(s_{\alpha}(\theta_1),\ldots,s_{\alpha}(\theta_r))$ is again a
sequence of rationally orthogonal roots with $\sum_i
s_{\alpha}(\theta_i)^\vee > \sum_i \theta_i^\vee$. The $W$-orbit of
$\sum_i \theta_i^\vee$ being bounded, this process must terminate
with some sequence $(\theta_i)$ such that the sum $\sum_i
\theta_i^\vee$ is dominant.
\end{proof}

\begin{remark}
\label{R2} For $(\theta_i)$ such that $\sum_i \theta_i^\vee$ is
dominant, if we label the vertex $\alpha$ of the Dynkin diagram of
$\mathfrak{g}$ with the  number $\alpha(\sum_i\theta_i^\vee)$, then
we obtain a weighted Dynkin diagram which is precisely the one
parametrizing the nilpotent orbit $\mathcal{O}_e$ in the sense of
\cite[\S3.5]{Collingwood-McGovern}. In particular the coweight
$\sum_i\theta_i^\vee$ is independent of the sequence $(\theta_i)$
involved in the statement.
\end{remark}

We now give a quite precise description of when rational
orthogonality differs from orthogonality. We write $\Phi \supset
X_r$ to mean that the root system of $\Phi$ contains a root
subsystem of type $X_r$.

\begin{lemma}
\label{rational_combination} Let $\beta_1,\ldots,\beta_k$ be a set
of orthogonal roots and let $n_1,\ldots,n_k \in \Q \setminus \{0\}$
be such that $\sum n_i \beta_i \in \Phi$. Then one of the following
holds:
\begin{itemize}
 \item[($D_4$)] $k=4$, all the roots $\beta_i$ have the same length, $\frac12(\beta_1+\beta_2+\beta_3+\beta_4)$ is a root, and $\Phi \supset D_4$.
 \item[($B_3$)] $k=3$ and, up to reordering the roots $\beta_i$, we have that $\beta_1$ and $\beta_2$ are long and $\beta_3$ is short,
 $\frac12(\beta_1+\beta_2+2\beta_3)$ is a long root, and $\Phi \supset B_3$.
 \item[($C_3$)] $k=3$ and, up to reordering the roots $\beta_i$, we have that $\beta_1$ and $\beta_2$ are short and $\beta_3$ is long,
 $\frac12(\beta_1+\beta_2+\beta_3)$ is a short root, and $\Phi \supset C_3$.
 \item[($B_2$ long)] $k=2$, $\beta_1$ and $\beta_2$ are short, $\beta_1+\beta_2$ is a long root, and $\Phi \supset B_2$.
 \item[($B_2$ short)] $k=2$, $\beta_1$ and $\beta_2$ are long, $\frac12(\beta_1+\beta_2)$ is a short root, and $\Phi \supset B_2$.
 \item[($G_2$ both)] $k=2$, up to reordering the roots $\beta_i$, $\beta_1$ is long and $\beta_2$ is short,
 $\frac12(\beta_1+3\beta_2)$ is a long root, $\frac12(\beta_1+\beta_2)$ is a short root, and $\Phi \supset G_2$.
 \item[($A_1$)] $k=1$ and $n_1=\pm 1$.
\end{itemize}
\end{lemma}

\begin{pe}
 reformulation de la remarque ci-dessous et réécriture de quelques points dans la preuve comme indiqué dans le draft annoté.
\end{pe}
\begin{remark}
\label{rem:5cases} The seven cases above can occur. In the list
given below, we use the standard numbering
$(\alpha_1,\ldots,\alpha_k)$ of the simple roots and we denote by
$\theta$ the highest root. Moreover, in each case, if we observe that a linear combination
$\sum q_i \beta_i$ is a root, then we let $e$ denote the nilpotent element $\sum e_{\beta_i}$.
\begin{itemize}
\item[($D_4$)] in type $D_4$, $\frac12(\theta+\alpha_1+\alpha_3+\alpha_4) \in \Phi$. The height of $e$ is $4$.
\item[($B_3$)] in type $B_3$, $\frac12(\theta+\alpha_1+2\alpha_3) \in \Phi$. The height of $e$ is $4$.
\item[($C_3)$] in type $C_3$, $\frac12(\alpha_2+(\alpha_2+\alpha_3)+\theta) \in \Phi$. The height of $e$ is $2$.
\item[($B_2$ long)] in type $B_2$, $\alpha_2+(\alpha_1+\alpha_2) \in \Phi$. The height of $e$ is $2$.
\item[($B_2$ short)] in type $B_2$, $\frac12(\theta+\alpha_1)=\alpha_1+\alpha_2 \in \Phi$. The height of $e$ is $2$.
\item[($G_2$ both)] in type $G_2$, $\frac12(\theta+3\alpha_1)=3\alpha_1+\alpha_2 \in \Phi$ and
 $\frac12(\theta+\alpha_1)=2\alpha_1+\alpha_2 \in \Phi$. The height of $e$ is $4$.
\end{itemize}
Note also that, if $\beta_1,\beta_2,\beta_3$ is a set of three
orthogonal roots which satisfy ($B_3$), resp. ($C_3$), then
$\beta_1,\beta_2$ satisfy ($B_2$ short), resp. ($B_2$ long).
\end{remark}
\begin{proof}[Proof of Remark \ref{rem:5cases}]
We justify the value of the height of $e$ indicated in each case.

In cases ($D_4$), ($B_3$), ($B_2$ short) and ($G_2$ both),
no integral combination of the roots $\beta_i$ is a root. It follows
that we can find an $\fsl_2$-triple $(e,h,f)$, whose element $h$ is
the sum of the coroots. In case $(D_4)$, we have
$h=\theta^\vee+\alpha_1^\vee+\alpha_3^\vee+\alpha_4^\vee$. This is
equal to
$2\varpi_1^\vee-2\varpi_2^\vee+2\varpi^\vee_3+2\varpi^\vee_4$, which
is in the $W$-orbit of $2\varpi_2^\vee$. The height of $e$ is then
the value of this coweight on the highest root, namely
$\scal{\theta,2\varpi_2^\vee}=4$.

The other cases are similar. In case $(B_3)$,
$h=2\varpi_1^\vee-2\varpi_2^\vee+2\varpi_3^\vee$, which is
equivalent to $2\varpi_2^\vee$. In case ($B_2$ short), $h=2\varpi_1^\vee$. In case
($G_2$ both), we have $\beta_1=\theta$ and $\beta_2=\alpha_1$. Then
$h=\beta_1^\vee + \beta_2^\vee = 2 \varpi_1^\vee - 2 \varpi_2^\vee$,
and $s_2(h)=2\varpi_2^\vee$. Thus $e$ has height
$\scal{\theta,2\varpi_2^\vee}=4$.

The cases ($C_3$) and ($B_2$ long) are addressed in Lemmas \ref{typeC} and \ref{typeB}.
\end{proof}

\pgfkeys{/Dynkin diagram,edge length=7mm}

\begin{proof}[Proof of Lemma \ref{rational_combination}]
Let $\beta_1,\ldots,\beta_k\in\Phi$ and $n_1,\ldots,n_k \in \Q
\setminus \{0\}$ be as in the statement of the lemma. Let $\beta =
\sum n_i \beta_i$. If $i \in \{ 1, \ldots, k \}$, then
$s_{\beta_i}(\beta)=-n_i \beta_i + \sum_{j \neq i} n_j \beta_j \in
\Phi$, so we may assume that $n_i >0$ for all $i$.

We consider the set $R=\{\beta_1,\ldots,\beta_k,-\beta\}$ and the
matrix $A=(\scal{\alpha,\gamma^\vee})_{\alpha,\gamma \in R}$. By
\cite[Lemma 5.2]{gandini_isotropy}, this matrix is a generalized
Cartan matrix of finite or affine type, and $-\beta$ is connected to
all the other roots since $\scal{\beta,\beta_i^\vee}=2n_i>0$ for all
$i$. Moreover, we have
$n_i=\frac12\scal{\beta,\beta_i^\vee}\in\frac{1}{2}\Z$.

If $A$ has size $2$, then $\beta$ is a scalar multiple of $\beta_1$,
and we are in case ($A_1$) of the statement.

For $A$ being of size $k\geq 3$, we use the classification tables
given at the end of chapter 4 in \cite{kac}, which comprise twelve
remaining cases that may fit the configuration of the root system.
\begin{enumerate}
 \item $A$ is of type $D_4^{(1)}$ and the Dynkin diagram of $A$ is
 \dynkin[labels={\beta_1,\beta_2,-\beta,\beta_3,\beta_4}] D[1]4.
 Then $\beta=\frac12(\beta_1+\beta_2+\beta_3+\beta_4)$, and $\beta_2,-\beta,\beta_3,\beta_4$ generate a subsystem
 of type $D_4$, so $\Phi \supset D_4$. We are in case ($D_4$) of the statement.
 \item $A$ is of type $B_3^{(1)}$ and the Dynkin diagram of $A$ is
 \dynkin[labels={\beta_1,\beta_2,,\beta_3},labels*={,,-\beta,},edge length=7mm] B[1]3.
 Then $\beta=\frac12(\beta_1+\beta_2+2\beta_3)$, and $\beta_2,-\beta,\beta_3$ generate a subsystem of type $B_3$. Moreover,
 $\beta_1,\beta_2,\beta$ are long and $\beta_3$ is short. We are in case ($B_3$).
 \item $A$ is of type $A_5^{(2)}$ and the Dynkin diagram of $A$ is
 \dynkin[labels={\beta_1,\beta_2,-\beta,\beta_3}] A[2]5.
 Then $\beta=\frac12(\beta_1+\beta_2+\beta_3)$, and $\beta_2,-\beta,\beta_3$ generate a subsystem of type $C_3$. Moreover, $\beta_1,\beta_2,\beta$ are short and $\beta_3$ is long. We are in case ($C_3$).
 \item $A$ is of type $D_3^{(2)}$ and the Dynkin diagram of $A$ is
 \dynkin[labels={\beta_1,-\beta,\beta_2}] D[2]3.
 Then $\beta=\beta_1+\beta_2$; $\beta_1$ and $\beta_2$ have the same length and $\beta$ is longer. Moreover,
 $-\beta,\beta_2$ generate a root subsystem of type $B_2$. We are in case ($B_2$ long).
 \item $A$ is of type $C_2^{(1)}$ and the Dynkin diagram of $A$ is
 \dynkin[labels={\beta_1,-\beta,\beta_2}] C[1]2.
 Then $\beta=\frac12(\beta_1+\beta_2)$; $\beta_1$ and $\beta_2$ have the same length and $\beta$ is  shorter. Again, $\beta_2,-\beta$ generate a subsystem of type $B_2$. We are in case ($B_2$ short).
 \item $A$ is of type $G_2^{(1)}$ and the Dynkin diagram of $A$ is
 \dynkin[labels={\beta_1,-\beta,\beta_2}] G[1]2.
 Then $\beta=\frac12(\beta_1+3\beta_2)$; $\beta_1,\beta$ are long roots, while $\beta_2$ is a short root as well as $\frac{1}{2}(\beta_1+\beta_2)$, since
$s_{\beta}(\beta_2)=\beta_2-\beta=\frac12(-\beta_1-\beta_2)$.
Moreover, $-\beta,\beta_2$ generate a subsystem of type $G_2$. We
are in case ($G_2$ both).
 \item $A$ is of type $D_4^{(3)}$ and the Dynkin diagram of $A$ is
 \dynkin[labels={\beta_1,-\beta,\beta_2}] D[3]4.
 Then $\beta=\frac12(\beta_1+\beta_2)$; $\beta_1,\beta$ are short roots while $\beta_2$ is a long root as well as $\frac{1}{2}(3\beta_1+\beta_2)$, since $s_\beta(\beta_2)=\frac12(-3\beta_1-\beta_2)$. Again, $\beta_2,-\beta$ generate a subsystem
 of type $G_2$, and we are in case ($G_2$ both).
\item $A$ is of type  $A_4^{(2)}$ and the Dynkin diagram of $A$ is
 \dynkin[labels={\beta_1,-\beta,\beta_2}] A[2]4.
Then the three roots $\beta_1,\beta,\beta_2$ have pairwise distinct
lengths. This is a contradiction. This case cannot occur.
\end{enumerate}
We now consider the cases where $A$ is of finite type.
\begin{enumerate}\setcounter{enumi}{8}
 \item $A$ is of type $D_4$ and the Dynkin diagram of $A$ is
 \dynkin[labels={\beta_1,,\beta_2,\beta_3},labels*={,-\beta,,},edge length=7mm] D4.
 Then the four roots $\beta_1,-\beta,\beta_2,\beta_3$ have the same length, but $\beta=\frac12(\beta_1+\beta_2+\beta_3)$. This is a contradiction.
 This case does not occur.
 \item $A$ is of type $B_3$ and the Dynkin diagram of $A$ is \dynkin[labels={\beta_1,-\beta,\beta_2}] B3.
 Then $\beta=\frac12(\beta_1+2\beta_2)$ so $||\beta||^2= \frac14 ||\beta_1||^2 + ||\beta_2||^2 = \frac32 ||\beta_2||^2$.
 This is again a contradiction, since we should have $||\beta||^2=2 ||\beta_2||^2$.
 \item $A$ is of type $C_3$ and the Dynkin diagram of $A$ is \dynkin[labels={\beta_1,-\beta,\beta_2}] C3.
Then $\beta=\frac12(\beta_1+\beta_2)$, so
$\|\beta\|^2=\frac34 \|\beta_2\|^2$,
 and once again the lengths do not match.
 \item $A$ is of type $A_3$ and the Dynkin diagram of $A$ is \dynkin[labels={\beta_1,-\beta,\beta_2}] A3.
 Then $\beta=\frac12(\beta_1+\beta_2)$, which is again a contradiction since the three roots should have the same length.
\end{enumerate}
\end{proof}

It appears from Remark \ref{rem:5cases} that, if $(\theta_k)$ is a
sequence of orthogonal roots which are not rationally orthogonal,
the orbit of the nilpotent element $\sum e_{\theta_k}$ may
nevertheless be spherical. We now give a full characterization of
this property.

\begin{proposition}
\label{P.orthogonal.spherical} Assume that $G$ is simple. Let
$\Theta=(\theta_1,\ldots,\theta_r)$ be a sequence of orthogonal
roots and, for every $k\in\{1,\ldots,r\}$, let
$e_{\theta_k}\in\fg_{\theta_k}\setminus\{0\}$. Then, the orbit
$G\cdot e$ of the element $e=\sum_{k=1}^r e_{\theta_k}$ is spherical
except exactly in the following cases (where we refer to the
conditions listed in Lemma \ref{rational_combination}).
\begin{enumerate}
    \item $G$ has type $D$ or $E_i$ ($i\in\{6,7,8\}$), and $\Theta$ contains four roots which
    satisfy condition {\rm ($D_4$)};
    \item $G$ has type $B$ or $F_4$, and $\Theta$ contains three roots which satisfy condition {\rm ($B_3$)} or two
    pairs of roots which satisfy condition {\rm ($B_2$ short)} with no common element;
    \item $G$ has type $G_2$, and $\Theta$ contains one pair of roots which satisfy condition {\rm ($G_2$ both)}.
\end{enumerate}
\end{proposition}

\begin{proof}
If $G$ has type $A$, then $\theta_1,\ldots,\theta_r$ are always
rationally orthogonal, and the result follows from Proposition
\ref{P.orbits}. In the other simply laced types $D$, $E_6$, $E_7$,
$E_8$, the roots $\theta_1,\ldots,\theta_r$ are rationally
orthogonal except if four of them satisfy condition ($D_4$), in
which case $e$ will have height $\geq 4$ in view of Remark
\ref{rem:5cases}. In type $G_2$, the roots are rationally orthogonal
unless they satisfy ($G_2$ both), in which case the claim that
$G\cdot e$ is not spherical again follows from Remark
\ref{rem:5cases}. Finally, when $G$ has type $C$, $B$, or $F_4$, the
result respectively follows from Lemmas \ref{typeC}, \ref{typeB}, or
\ref{typeF} below.
\end{proof}

In the following sequence of lemmas, we consider the notation of
Proposition \ref{P.orthogonal.spherical}.

\begin{lemma}    \label{lemm:oneroot}
 Let $i,j \in \{1,\ldots,r\}$ be such that $\theta_i-\theta_j \in \Phi$. Then, for every $k \neq j$, $\theta_k + \theta_i - \theta_j \not \in \Phi$.
\end{lemma}
\begin{proof}
Let $\beta=\theta_i-\theta_j \in \Phi$ and let $k \neq i,j$. Since
$\scal{\theta_k,\theta_i^\vee}=0$ and
$\scal{\theta_k,\theta_j^\vee}=0$, we have
$\scal{\theta_k,\beta^\vee}=0$. Since $\beta$ is a long root, it
follows that $\theta_k + \beta \not \in \Phi$. For length reasons,
we must also have $2\theta_i - \theta_j \notin \Phi$.
\end{proof}

\begin{pe}
Tout à la fin, $2\theta_i - \theta_j \notin \Phi$. 
\end{pe}

\begin{lemma}    \label{lemm:exp}
Assume that there are $i,j \in \{1,\ldots,r\}$ such that
$\theta_i-\theta_j \in \Phi$. Then
$$\sum_{k=1}^r e_{\theta_k} \in G \cdot \sum_{k \neq i} e_{\theta_k}\,.$$
\end{lemma}
\begin{proof}
Let $\beta = \theta_i-\theta_j$ and let $e_\beta \in
\fgg_\beta\setminus\{0\}$. Let $g \in G$ be the exponential of
$e_\beta$. Then, by Lemma \ref{lemm:oneroot}, $g \cdot e_{\theta_k}
= e_{\theta_k}$ for $k \neq j$, and $g \cdot e_{\theta_j} =
e_{\theta_i} + e_{\theta_j}$ (up to multiplying $e_\beta$ by some
suitable scalar). The lemma follows.
\end{proof}

\begin{remark}
\label{rem:oneshortroot} (a) The condition that
$\theta_i-\theta_j\in\Phi$ of Lemma \ref{lemm:exp} can as well be
replaced by $\theta_i+\theta_j\in\Phi$. Indeed, since $\theta_i$ and
$\theta_j$ are orthogonal roots, then $\theta_i-\theta_j$ is a root
if and only if $\theta_i+\theta_j=s_{\theta_j}(\theta_i-\theta_j)$
is a root.

(b) It follows from Lemma \ref{lemm:exp} that, for studying whether
the $G$-orbit of $e=\sum e_{\theta_k}$ is spherical or for computing
the height of $e$, we may assume that no pair $(\theta_i,\theta_j)$
satisfies case ($B_2$ long) of Lemma \ref{rational_combination}.
Indeed, if case ($B_2$ long) occurs within the sequence
$(\theta_k)$, then this sequence may be replaced by a subsequence
for which ($B_2$ long) does not occur, and which yields a generator
of the same nilpotent $G$-orbit.
\end{remark}

\begin{pe}
 ajout ``of height $2$''
\end{pe}
\begin{lemma}
\label{typeC} If $G$ has type $C$, then $e=\sum e_{\theta_k}$ is
always spherical of height $2$.
\end{lemma}

\begin{proof}
By Remark \ref{rem:oneshortroot}, we may assume that ($B_2$ long)
does not occur within the sequence $(\theta_k)$. Up to reordering
the roots $\theta_k$, we may set an integer $l\in\{0,\ldots,r\}$
such that $\theta_1,\ldots,\theta_l$ are long and
$\theta_{l+1},\ldots,\theta_r$ are short. Up to the action of the
Weyl group, we may assume that $\theta_1=2\epsilon_1$,
$\theta_2=2\epsilon_2$, and so on up to $\theta_l=2\epsilon_l$. If
$l<r$, we may then assume that
$\theta_{l+1}=\epsilon_{l+1}-\epsilon_{l+2}$. Since
$(\theta_{l+1},\theta_{l+2})$ does not satisfy ($B_2$ long),
$\theta_{l+2}$ is of the form $\epsilon_a\pm \epsilon_b$ with $a,b >
l+2$, and we may assume that
$\theta_{l+2}=\epsilon_{l+3}-\epsilon_{l+4}$, and similarly for the
next short roots. We can then check readily that
$(\mathrm{ad}\,e)^3=0$ and invoke Proposition \ref{P4}.
\end{proof}

\begin{lemma}
\label{typeB} Assume that $G$ has type $B$. If one of the following
conditions is satisfied:
\begin{enumerate}
\item All the roots $\theta_1,\ldots,\theta_r$ are short;
\item All the roots $\theta_1,\ldots,\theta_r$ are long and
\begin{itemize}
\item $r=2$, or
\item no pair $(\theta_i,\theta_j)$ satisfies condition {\rm ($B_2$ short)} of Lemma \ref{rational_combination},
\end{itemize}
\end{enumerate}
then $e$ has height $2$.

If one of the following conditions is satisfied:
\begin{enumerate} \setcounter{enumi}{2}
\item There are at least one short root and one long root among $\theta_1,\ldots,\theta_r$, and no pair $(\theta_i,\theta_j)$ satisfies {\rm ($B_2$ short)};
\item $r\geq 3$, all the roots $\theta_1,\ldots,\theta_r$ are long, and there is exactly one pair $(\theta_i,\theta_j)$ which satisfies {\rm ($B_2$ short)},
\end{enumerate}
then $e$ has height $3$.

In any other case, the height of $e$ is at least $4$. In particular,
$e$ is spherical if and only if one of the above conditions {\rm
(1)--(4)} is satisfied.
\end{lemma}
\begin{proof}
To prove the result, we use \cite[Theorem 2.3]{Panyushev2}.
According to this result, an element $x\in \mathfrak{so}(n) \subset
\mathfrak{sl}(n)$:
\begin{itemize}
\item has height $2$ if $x^2=0$ or ($x^3=0$ and $\mathrm{rank}\,x=2$);
\item has height $3$ if $x^2\not=0$, $x^3=0$, $\mathrm{rank}\,x\geq 3$, and $\mathrm{rank}\,x^2=1$;
\item has height at least $4$ otherwise.
\end{itemize}

If all the roots $\theta_1,\ldots,\theta_r$ are short, then $e^3=0$
and $\mathrm{rank}\,e=2$; so the height of $e$ is $2$ in this case.

If all the roots $\theta_k$ are long but no pair
$(\theta_i,\theta_j)$ satisfies ($B_2$ short), then we have $e^2=0$,
hence $e$ has height $2$. If $r=2$ and the two roots are as in
($B_2$ short), then $e^3=0$ and $\mathrm{rank}\,e=2$, so the height
of $e$ is again $2$.

If there are at least one short root and one long root, then $e^2
\neq 0$ and $\mathrm{rank}\,e \geq 4$. Moreover, if no pair
$(\theta_i,\theta_j)$ satisfies ($B_2$ short), then we have
$\mathrm{rank}\,e^2=1$, and so $e$ has height 3 in this case. If at
least one pair $(\theta_i,\theta_j)$ satisfies ($B_2$ short), then
$\mathrm{rank}\,e^2\geq 2$, so that $e$ has height $\geq 4$.

Finally, assume that all the roots $\theta_k$ are long, $r\geq 3$,
and exactly two of the roots are as in ($B_2$ short). Then $e^3=0$,
$\mathrm{rank}\,e^2 =1$, and $\mathrm{rank}\,e \geq 4$, hence $e$
has height $3$. If there are at least two pairs of roots that
satisfy ($B_2$ short), then we have $\mathrm{rank}\,e^2\geq 2$,
hence $e$ has height at least 4. The last claim of the statement
follows from Proposition \ref{P4}.
\end{proof}

\begin{lemma}
\label{typeF} If $G$ has type $F_4$, assume that {\rm ($B_2$ long)}
does not occur within the sequence $(\theta_k)$. Then $e$ is
spherical if and only if $r \leq 2$ or ($r=3$ and all the roots are
long).
\end{lemma}

\begin{proof}
One feature of type $F_4$ is that, whenever $\theta_i,\theta_j$ are
orthogonal short roots, the sum $\theta_i+\theta_j$ is always a long
root, that is, $(\theta_i,\theta_j)$ satisfies ($B_2$ long).
\begin{pe}
 explication ajoutée:
\end{pe}
Indeed, we may assume, up to the action of the Weyl group, that $\theta_i$
is the heighest short root, namely
$\theta_i = \alpha_1 + 2 \alpha_2 + 3 \alpha_3 + 2 \alpha_2 = \varpi_4$,
and that $\theta_j=\alpha_3$.
Since
we assume that the sequence $(\theta_k)$ contains no pair that
satisfies ($B_2$ long), it follows that $(\theta_k)$ contains at
most one short root.

If $r=1$, then clearly $e$ is spherical. If $r=2$ and
$\theta_1,\theta_2$ are both long, then up to the action of the Weyl
group we may assume that $\theta_1=\theta=\rootfour{2}{3}{4}{2}$ is
the highest root and $\theta_2=\alpha_2+2\alpha_3+2\alpha_4$. Then,
since no integral combination of $\theta_1$ and $\theta_2$ is a
root, an $\fsl_2$-triple for $e$ is given by $h=\theta_1^\vee +
\theta_2^\vee$ and $f=f_{\theta_1}+f_{\theta_2}$ for some
$f_{\theta_1}\in\fg_{-\theta_1}$, $f_{\theta_2}\in\fg_{-\theta_2}$.
One computes that $\theta_1^\vee + \theta_2^\vee=\varpi_4^\vee$ is
dominant, so the height of $e$ is $\theta(\varpi_4^\vee)=2$.

If $r=2$ and $\theta_1$ is long and $\theta_2$ is short, once again,
the action of $W$ allows us to assume that
$\theta_1=\rootfour{2}{3}{4}{2}$ and
$\theta_2=\alpha_2+2\alpha_3+\alpha_4$. The same argument holds, and
$e$ belongs to an $\fsl_2$-triple $(e,h,f)$ such that
$h=\theta_1^\vee+\theta_2^\vee$. Then, we can see that
$h=-\varpi_1^\vee+\varpi_3^\vee=s_1s_2(\varpi_2^\vee)$. Hence, the
height of $e$ is $\theta(\varpi_2^\vee)=3$.

If $r=3$, since there is at most one short root, we may assume that
$\theta_1$ and $\theta_2$ are long, and that we have again
$\theta_1=\rootfour{2}{3}{4}{2}$ and
$\theta_2=\alpha_2+2\alpha_3+2\alpha_4$. Then
$\theta_1^\vee=\varpi_1^\vee$ and
$\theta_2^\vee=-\varpi_1^\vee+\varpi_4^\vee$, so $\theta_3$ must be
a linear combination of $\alpha_2$ and $\alpha_3$. Up to the action
of $W$, there are two cases: if $\theta_3$ is short, then we may
assume that $\theta_3=\alpha_2+\alpha_3$, and then
$\theta_1^\vee+\theta_2^\vee+\theta_3^\vee=-2\varpi_1^\vee+2\varpi_2^\vee=s_1(2\varpi_1^\vee)$.
Thus $e$ has height $\theta(2\varpi_1^\vee)=4$, and it is not
spherical. If $\theta_3$ is long, we may assume that
$\theta_3=\alpha_2+2\alpha_3$. Then $\theta_1^\vee + \theta_2^\vee +
\theta_3^\vee=-\varpi_1^\vee+\varpi_3^\vee=s_1s_2(\varpi_2^\vee)$,
so $e$ has height $3$ and it is spherical in this case.

If $r=4$, we know that there are at least three long roots. So, as
in the previous paragraph, we may assume that
$\theta_1=\rootfour{2}{3}{4}{2}$,
$\theta_2=\alpha_2+2\alpha_3+2\alpha_4$ and
$\theta_3=\alpha_2+2\alpha_3$. Then $\theta_4$ must be equal to
$\pm\alpha_2$, say $\alpha_2$. We recover the coweight
$\theta_1^\vee+\theta_2^\vee+\theta_3^\vee+\theta_4^\vee=-2\varpi_1^\vee+2\varpi_2^\vee=s_1(2\varpi_1^\vee)$.
Hence $e$ has height $4$, so it is not spherical.
\end{proof}

A convenient way to produce orthogonal roots is to use Harish-Chandra chain cascade of roots in the version defined by Kostant \cite[\S1]{kostant}. However, in
this way we get a sequence of strongly orthogonal roots \cite[Lemma
1.6]{kostant} but not necessarily a sequence of rationally
orthogonal roots. For example, in type $E_6$,
$(\theta,\alpha_1+\alpha_3+\alpha_4+\alpha_5+\alpha_6,\alpha_3+\alpha_4+\alpha_5,\alpha_4)$
is such an example, corresponding to case ($D_4$) in Lemma
\ref{rational_combination}. The examples for cases ($B_2$ short) and
($G_2$ both) given in Remark \ref{rem:5cases} are also chain
cascades, and one can check that these three cases are the only
cases of Lemma \ref{rational_combination} that can occur in a chain
cascade. Our final remark in this line of ideas is the following:

\begin{remark}
\label{remark_gandini} Let $(\theta_k)$ be a sequence obtained by
chain cascade. If the nilpotent element $\sum e_{\theta_k}$ has
height $2$, then the element $\sum \theta_k^\vee$ is dominant, by
\cite[Proposition 3.10]{gandini}.
\end{remark}

\begin{example}
\label{exam:nil-elements} {\rm (a)} Assume that
$G=\mathrm{SL}_n(\K)$. A nilpotent matrix $e\in\mathfrak{sl}_n(\K)$
belongs to a spherical nilpotent orbit if and only if $e^2=0$ (see
\cite{Panyushev1}). For every
$r\in\{0,\ldots,\lfloor\frac{n}{2}\rfloor\}$, the set
$\mathcal{O}^{(r)}:=\{e\in\mathfrak{sl}_n(\mathbb{K}): e^2=0,\
\mathrm{rank}\,e=r\}$ consists of a single $\SL_n(\K)$-orbit. Let
$\Phi=\{\epsilon_i-\epsilon_j:1\leq i\not=j\leq n\}$ be the usual
root system. Then for every permutation $\sigma$ of
$\{1,\ldots,r\}$, the roots
\[\theta_i^\sigma:=\epsilon_i-\epsilon_{n+1-\sigma(i)}\quad\mbox{(for $i=1,\ldots,r$)}\]
form a sequence of rationally orthogonal roots, and for every choice
of root vectors
$e_{\theta_i^\sigma}\in\mathfrak{g}_{\theta_i^\sigma}\setminus\{0\}$,
the element
\[e_\sigma:=\sum_{i=1}^re_{\theta_i^\sigma}\]
is a representative of $\mathcal{O}^{(r)}$ which is of the form
described in Proposition \ref{P.orbits}\,(b), i.e., the
corresponding coweight $\sum_{i=1}^r(\theta_i^\sigma)^\vee$ is
dominant. Thus the sequence of rationally orthogonal roots described
in Proposition \ref{P.orbits}\,(b) is not unique for each orbit.
Note however that the coweight $\sum_{i=1}^r(\theta_i^\sigma)^\vee$
is independent of $\sigma$, which agrees with Remark \ref{R2}. For
$\sigma=\mathrm{id}$, the sequence $(\theta_i^\mathrm{id})$ is a
chain cascade.

{\rm (b)} As another example, we give the sequences of roots
$(\theta_i)$ given by Proposition \ref{P.orbits} in the case of the
exceptional group of type $E_7$. In view of \cite[table
p.\,130]{Collingwood-McGovern}, $\mathfrak{g}$ has five nontrivial
spherical nilpotent orbits in this case, whose types are
respectively $A_1$, $2A_1$, $3A_1$, $3A_1$, $4A_1$. We use the
following numbering of the simple roots/vertices of the Dynkin
diagram:
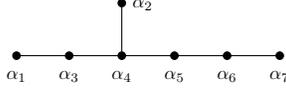
\[
\dynkin[labels={\alpha_1,\alpha_2,\alpha_3,\alpha_4,\alpha_5,\alpha_6,\alpha_7}]
E7
\]
For $i\in\{1,\ldots,7\}$, we set $s_i=s_{\alpha_i}$. Note that in a
simply laced case (as $E_7$), two roots are orthogonal if and only
if they are strongly orthogonal.

Let $\theta_1$ be the highest root, which is equal to the
fundamental weight $\varpi_1$. Thus, we have $\theta_1^\vee =
\varpi_1^\vee$. The root vector $e_{\theta_1}$ belongs to the orbit
of type $A_1$ (which is the minimal nilpotent orbit of
$\mathfrak{g}$).

Let $\theta_2$ be the highest root which is orthogonal to
$\theta_1$, namely
$\alpha_2+\alpha_3+2\alpha_4+2\alpha_5+2\alpha_6+\alpha_7$. It is
also equal to $-\varpi_1+\varpi_6$, so that $\theta_1^\vee +
\theta_2^\vee = \varpi_6^\vee$. The vector
$e_{\theta_1}+e_{\theta_2}$ belongs to the orbit of type $2A_1$.

The orthogonal of $\theta_1$ and $\theta_2$ consists of roots with
no term in $\alpha_1$ nor $\alpha_6$. Since this orthogonal is
disconnected, we have three choices for the third root in the chain
cascade according to Kostant's algorithm \cite{kostant}. One is
$\theta_3''=\alpha_7=2\varpi_7-\varpi_6$. By Lemma
\ref{rational_combination}, $(\theta_1,\theta_2,\theta_3'')$ is a
set of rationally orthogonal roots, and $\theta_1^\vee +
\theta_2^\vee + (\theta_3'')^\vee = 2 \varpi_7^\vee$. The vector
$e_{\theta_1}+e_{\theta_2}+e_{\theta_3''}$ belongs to the nilpotent
orbit  labeled $(3A_1)''$ in \cite[table
p.\,130]{Collingwood-McGovern}.

Note that up to now, the heights of the nilpotent elements were $2$
and, as predicted by Remark \ref{remark_gandini}, the coweights
$\sum_i \theta_i^\vee$ were dominant. We now consider some nilpotent
elements of height $3$ and $\sum_i \theta_i^\vee$ will not be
dominant.

The second choice for the third root in the chain cascade is
$\theta_3' = \alpha_2 + \alpha_3 + 2 \alpha_4 + \alpha_5$. Then, we
have $\theta_1^\vee + \theta_2^\vee + (\theta_3')^\vee =
-\varpi_1^\vee+\varpi_4^\vee$, which is not a dominant coweight. In
this case, we  consider the sequence
$s_3s_1(\theta_1),s_3s_1(\theta_2),s_3s_1(\theta_3')$, which is also
a set of rationally orthogonal roots, and which satisfies
$s_3s_1(\theta_1)^\vee + s_3s_1(\theta_2)^\vee +
s_3s_1(\theta_3')^\vee = \varpi_3^\vee$. The  vector
$e_{s_3s_1(\theta_1)}+e_{s_3s_1(\theta_2)}+e_{s_3s_1(\theta_3')}$
belongs to the orbit labeled $(3A_1)'$ in \cite[table
p.\,130]{Collingwood-McGovern}.

Finally, the three possible chain cascades of length 4 are
$(\theta_1,\theta_2,\theta_3',\alpha_i)$ with $i\in\{2,3,5\}$. We
have
$\theta_1^\vee+\theta_2^\vee+(\theta_3')^\vee+\alpha_2^\vee=-\varpi_1^\vee+2\varpi_2^\vee$
(which is not a dominant coweight). Letting $s=s_7s_6s_5s_4s_3s_1$,
we get
$s(\theta_1^\vee)+s(\theta_2^\vee)+s((\theta'_3)^\vee)+s(\alpha_2^\vee)=\varpi_2^\vee+\varpi_7^\vee$
(which is a dominant coweight). The element
$e_{s(\theta_1)}+e_{s(\theta_2)}+e_{s(\theta'_3)}+e_{s(\alpha_2)}$
belongs to the orbit labeled $4A_1$.

In the same way, we have
$\theta_1^\vee+\theta_2^\vee+(\theta_3')^\vee+\alpha_5^\vee=-\varpi_1^\vee+2\varpi_5^\vee-\varpi_6^\vee=s_1s_3s_4s_2s_6s_7(\varpi_2^\vee+\varpi_7^\vee)$,
hence we recover the same orbit labeled $4A_1$.

The situation is different for the last possible chain cascade
$(\theta_1,\theta_2,\theta_3',\alpha_3)$. These roots are not
rationally orthogonal as they satisfy condition $(D_4)$ of Lemma
\ref{rational_combination}. In fact, in this case, we have
$\theta_1^\vee+\theta_2^\vee+(\theta_3')^\vee+\alpha_3^\vee=-2\varpi_1^\vee+2\varpi_3^\vee=s_1(2\varpi_1^\vee)$,
hence the nilpotent element
$e_{\theta_1}+e_{\theta_2}+e_{\theta_3'}+e_{\alpha_3}$ belongs to
the orbit labeled $A_2$ in \cite[table
p.\,130]{Collingwood-McGovern} (which is of height $4$).

\end{example}

\section{Symmetric subgroup associated to a spherical nilpotent orbit}

\label{SI.4}

In this section, we assume that the orbit $G\cdot e$ is spherical
and relate the subgroups $L_Z^0$ and $L_Z$ of $Z$ to a symmetric
subgroup of $L$. We use the notation introduced in Section
\ref{SI.1}.

\begin{definition}
\label{D.sigma} Let $e\in\mathfrak{g}$ be a nilpotent element and
let $\mathfrak{s}=(e,h,f)$ be a standard triple. As in Section
\ref{SI.1}, we consider a cocharacter $\tau:\K^*\to G$ such that
$\tau'(1)=h$. We assume that $h$ belongs to the Lie algebra
$\mathfrak{t}=\mathrm{Lie}(T)$ of the standard torus $T\subset G$,
so that $\tau(\K^*)$ is contained in $T$.

We consider the Lie algebra $\mathfrak{sl}_2(\mathfrak{s})$ linearly
generated by $e$, $h$, and $f$, and we consider the subgroup
$\mathrm{SL}_2(\mathfrak{s})\subset G$ with Lie algebra
$\mathfrak{sl}_2(\mathfrak{s})$. The torus $\tau(\K^*)$ is a maximal
torus of $\SL_2(\fs)$.

We denote by $n_\mathfrak{s}\in \mathrm{SL}_2(\mathfrak{s})$ an
element in the normalizer of $\tau(\K^*)$, and not in $\tau(\K^*)$
(thus, a representative of the nontrivial element of the Weyl group
of $\mathrm{SL}_2(\mathfrak{s})$). By a standard calculation in
$\SL_2(\K)$, we get
\begin{equation}
\label{n-tau} n_\fs\tau(t)n_\fs^{-1}=\tau(t^{-1})\ \mbox{ for all
$t\in\K^*$}.
\end{equation}

We denote by $\sigma:G \to G$ the conjugation by $n_\mathfrak{s}$
and we use the same notation for the adjoint action of
$n_\mathfrak{s}$ on $\mathfrak{g}$.
\end{definition}

We make a first observation:
\begin{lemma}
\label{L.inclusion}
\begin{itemize}
\item[\rm (a)] The map $\sigma$ preserves the subgroup $L$; in fact, $\sigma:L\to L$ is an involution.
\item[\rm (b)] We have $Z\cap L\subset L^\sigma$ and $\fz_\fgg(e)\cap\fl\subset \fl^\sigma$.
\end{itemize}
\end{lemma}

\begin{proof}
{\rm (a)} In view of (\ref{n-tau}), $\sigma$ preserves the torus
$\tau(\K^*)$, hence $\sigma$ also preserves the centralizer of
$\tau(\K^*)$, which is precisely $L$. Since the Weyl group of
$\SL_2(\fs)$ has only two elements, we must have
$n_\fs^2\in\tau(\K^*)$, hence $n_\fs^2$ is contained in the center
of $L$. This implies that $\sigma^2(g)=g$ for all $g\in L$.

{\rm (b)} We have $Z\cap L=Z_G(e)\cap
Z_G(h)=Z_G(\mathfrak{sl}_2(\fs))$ (see \cite[Lemma
3.4.4]{Collingwood-McGovern}). Moreover
$Z_G(\mathfrak{sl}_2(\fs))=Z_G(\SL_2(\fs))$, hence $Z\cap
L=Z_G(\SL_2(\fs))$. Since $n_s\in \SL_2(\fs)$, this implies that
$Z\cap L\subset L^\sigma$. This inclusion yields the inclusion of
Lie subalgebras $\fz_\fgg(e)\cap\fl\subset \fl^\sigma$.
\end{proof}

\begin{remark}
\label{sigma-L-L} Whereas the map $\sigma:G\to G$ a priori depends
on the choice of $n_\fs$ within a $\tau(\K^*)$-coset, its
restriction $\sigma:L\to L$ is independent of this choice, because
$L=Z_G(\tau)$.
\end{remark}

\begin{proposition}
\label{P.symmetric-space} Let $e = \sum_{i=1}^r e_{\theta_i}$ be as
in Proposition \ref{P.orbits}\,{\rm (a)}; in particular its orbit
$G\cdot e$ is spherical. Then, we have $\fz_\fgg(e)\cap\fl =
\fl^\sigma$ and $L_Z^0=(L^\sigma)^0\subset L_Z\subset L^\sigma$.
\end{proposition}
\begin{proof}
Recall that $L$ is the centralizer of $\tau(\K^*)$. Thus,
$$\fl = \ft \oplus \bigoplus_{\substack{\alpha\in\Phi \\ \mathrm{s.t.}\,\scal{\alpha,h}=0}} \fg_\alpha.$$
We claim that
\begin{equation}
\label{T-preserved} \sigma(T)=T,\quad \mbox{i.e., $n_\fs$ belongs to
$N_G(T)$}.
\end{equation}
Indeed, we have already noted that $\sigma$ preserves the torus
$\tau(\K^*)$. Let $S=\bigcap_i\ker \theta_i$. Then $S\subset Z\cap
T\subset Z\cap L$, which implies that $\sigma$ fixes every element
of $S$ (by Lemma \ref{L.inclusion}). Thereby $\sigma$ preserves the
torus $S\tau(\K^*)$ and so it also preserves its centralizer. As
noted in the proof of Proposition \ref{P.orbits}, we have
$Z_G(S\tau(\K^*))=T$. Whence (\ref{T-preserved}).

Relation (\ref{T-preserved}) implies that $\sigma$ induces an
involution on the root system $\Phi$ (that we denote by the same
letter). In fact this action coincides with the action of the Weyl
group element
\[s_{\theta_1}\cdots s_{\theta_r}:\alpha\mapsto \alpha-\sum_i\scal{\alpha,\theta_i^\vee}\theta_i.\]

We have the following equality:
\begin{equation}
\label{E.l} \fl^\sigma = \bigoplus_\alpha \fg_\alpha^\sigma \oplus
\bigoplus_{\{\alpha,\beta\}} {(\fg_\alpha \oplus
\fg_{\beta})}^\sigma \oplus \ft^\sigma\ ,
\end{equation}
where the first sum is over the roots $\alpha$ of $L$ such that
$\sigma(\alpha)=\alpha$, and the second sum is over the pairs
$\{\alpha,\beta\}$ of roots of $L$ with $\beta = \sigma(\alpha)$.
Let $\alpha$ be a root of $L$. We consider two cases.

\smallskip

First, if $\sigma(\alpha)=\alpha$, then for all $i$ we have
$\scal{\alpha,\theta_i^\vee}=0$. Assume first that there exists $i$
such that $\alpha+\theta_i$ is a root. Then
$\alpha-\theta_i=s_{\theta_i}(\alpha+\theta_i)$ is also a root. We
claim that there can be only one such integer $i$. In fact, let us
assume to the contrary that there are two integers $i,j$ such that
$\alpha + \theta_i$ and $\alpha + \theta_j$ are roots. Then,
considering the $(\fsl_2(\theta_i) \times \fsl_2(\theta_j))$-module
generated by $\fg_\alpha$, we deduce that $\alpha+\theta_i+\theta_j$
is a root. But then, the square lengths of $\alpha,\alpha+\theta_i$
and $\alpha+\theta_i+\theta_j$ are three different numbers, a
contradiction.

Let $i$ be the unique integer such that $\alpha+\theta_i$ is a root.
It follows that the $\fsl_2(\theta_i)$-module generated by
$\fg_\alpha$ is isomorphic to the adjoint module $\fsl_2(\theta_i)$.
Thus the same holds for the corresponding $\SL_2(\theta_i)$-module.
By a direct computation in $\SL_2(\K)$, we deduce that the
restriction of the nontrivial element of the Weyl group of
$\SL_2(\theta_i)$ to $\fg_\alpha$ is $-\mathrm{id}_{\fg_\alpha}$.
Since the elements of the Weyl group of $\SL_2(\theta_j)$ act
trivially on $\fg_\alpha$ for $j \neq i$, we deduce that $\sigma$
acts as $-\mathrm{id}_{\fg_\alpha}$ on $\fg_\alpha$. Therefore,
$\fg_\alpha^\sigma = \{0\}$ in this case.

If for all integers $i$, $\alpha+\theta_i$ is not a root, we get
$\fg_\alpha \subset \fz_\fgg(e)$. In both cases, we deduce from Lemma
\ref{L.inclusion} that we have $\fz_{\fg}(e)\cap\fg_\alpha =
\fg_\alpha^\sigma$.

\smallskip

Second, if $\sigma(\alpha) \neq \alpha$, let us set $\beta =
\sigma(\alpha)$. First, assume that for all integers $i$, we have
$|\scal{\alpha,\theta_i^\vee}| \leq 1$. Then, the number of integers
$i$ such that $\scal{\alpha,\theta_i^\vee} \neq 0$ is $2$. Indeed,
it must be even since $\scal{\alpha,\theta_i^\vee}$ belongs to
$\{-1,0,1\}$ for all $i$, and $\sum_i\langle
\alpha,\theta_i^\vee\rangle=0$. Moreover, it must be less than $4$
because otherwise $\alpha$ and four such roots would generate an
infinite root subsystem of type $D_4^{(1)}$. Assume now that there
exists an integer $i$ such that $|\scal{\alpha,\theta_i^\vee}| \geq
2$. Then, since $\sum_j \scal{\alpha,\theta_j^\vee} = 0$, there are
two cases. Either there is another integer $j \neq i$ such that
$|\scal{\alpha,\theta_j^\vee}| \geq 2$, in which case the root
subsystem generated by $\theta_i,\alpha$ and $\theta_j$ is infinite,
which is absurd. Or there are at least two integers $j$ such that
$|\scal{\alpha,\theta_j^\vee}|=1$ and, once again, this is absurd.

We therefore have shown that there are exactly two integers $k$ such
that $\scal{\alpha,\theta_k^\vee} \neq 0$, and that for these two
integers we have $|\scal{\alpha,\theta_k^\vee}|=1$. We let $i,j$ be
the integers such that $\scal{\alpha,\theta_i^\vee} = 1$ and
$\scal{\alpha,\theta_j^\vee} = -1$. We have $\beta = \sigma(\alpha)
= \alpha - \theta_i + \theta_j$. Moreover, $s_{\theta_i}(\alpha) =
\alpha - \theta_i$ is a root, so $\alpha + \theta_i$ is not a root.
Note also that if $k\not=i,j$, then $\alpha+\theta_k$ cannot be a
root, otherwise $\alpha+\theta_k,-\theta_k,-\theta_i,\theta_j$ would
generate an infinite root subsystem of type $B_3^{(1)}$. Similarly,
$\beta + \theta_k$ is not a root unless $k=i$. It follows that
$$ [e,(\fg_\alpha \oplus \fg_{\beta})] = \fg_{\alpha + \theta_j} = \fg_{\beta+\theta_i}\, .$$
Thus, $\dim\fz_{\fg}(e)\cap(\fg_\alpha \oplus \fg_{\beta})=1$. On
the other hand, $(\fg_\alpha \oplus \fg_{\beta})^\sigma$ is also
$1$-dimensional. By Lemma \ref{L.inclusion}, we get $(\fg_\alpha
\oplus \fg_{\beta})^\sigma = \fz_{\fg}(e)\cap(\fg_\alpha \oplus
\fg_{\beta})$.

\smallskip

Finally, we claim that $\ft^\sigma = \fz_\fgg(e)\cap\ft$. In fact,
$\ft^\sigma$ is the orthogonal of all the roots $\theta_i$, because
$\sigma$ is the product of the reflections defined by
$\theta_1,\ldots,\theta_r$. Since $e_{\theta_i}$ has weight
$\theta_i$, $\fz_\fgg(e)\cap\ft$ is also the orthogonal of the roots
$\theta_i$. Whence the claimed equality.

\smallskip

Altogether we have shown the desired equality $\fz_\fgg(e)\cap\fl =
\fl^\sigma$. This equality implies $L_Z^0=(L^\sigma)^0$. The
inclusion $L_Z\subset L^\sigma$ is already noted in Lemma
\ref{L.inclusion}.
\end{proof}

\begin{remark}
The fact that $\fz_\fgg(e)\cap\fl$ is a symmetric subalgebra of $\fl$
is already shown in \cite[\S3.3]{Panyushev1} (with a different
proof). We also refer to \cite[Appendix B]{Bravi-Cupit-Foutou} where
the symmetric pairs $(L,L^\sigma)$ corresponding to the nilpotent
elements $e$ of height 3 are explicitly described.
\end{remark}

\begin{example}
\label{E4.5} {\rm (a)} We first consider the case of nilpotent
orbits in type $A_l$. Let $e$ be a nilpotent element of height $2$
and rank $r$ in the Lie algebra $\fsl_{l+1}$: this means that the
endomorphism $e$ satisfies $e^2=0$ and that the partition giving the
size of the Jordan blocks is $(2^r,1^{l+1-2r})$. The coweight $h$ in
an $\fsl_2$-triple $(e,h,f)$ is $\varpi_r^\vee+\varpi_{l+1-r}^\vee$
(see also Example~\ref{exam:nil-elements} (a)). Thus $L$ is the
standard Levi subgroup corresponding to the root subsystem of type
$A_{r-1}\times A_{r-1}\times A_{l-2r}$ generated by the simple roots
$\alpha_1,\ldots,\alpha_{r-1},\alpha_{r+1},\ldots,\alpha_{l-r},\alpha_{l+2-r},\ldots,\alpha_l$.

The element $e$ can be described as a sum of root vectors using the
procedure given in Proposition \ref{P.orbits}: it is in the orbit of
the sum $e_{\theta_1}+\cdots+e_{\theta_r}$ where
$\theta_i=\alpha_i+\cdots+\alpha_{l+1-i}$. Moreover, the involution
is given by the nontrivial Weyl group element of the subgroup with
Lie algebra linearly generated by $(e,h,f)$. This element acts on
the roots as the element $w=s_{\theta_1} \cdots s_{\theta_r}$ of the
Weyl group. Thus, we have
$w(\alpha_i)=\alpha_i-\theta_i+\theta_{i+1}=-\alpha_{l+1-i}$ if
$1\leq i\leq r-1$, and $w(\alpha_i)=\alpha_i$ if $r+1\leq i\leq
l-r$. In this way, we recover the known fact that $L$ is of type
$A_{r-1} \times A_{r-1} \times A_{l-2r}$ while $L^\sigma$ is of type
$A_{r-1} \times A_{l-2r}$ (with a diagonal embedding of $A_{r-1}$ in
the factor $A_{r-1} \times A_{r-1}$).

{\rm (b)} We now consider the case of the nilpotent orbit labeled
$3A_1$ in type $E_6$, which corresponds to the dominant coweight
$\varpi_4^\vee$ (see \cite[table p. 129]{Collingwood-McGovern}). A
representative of this orbit is obtained by chain cascade, by
letting $\theta_1 =
\alpha_1+2\alpha_2+2\alpha_3+3\alpha_4+2\alpha_5+\alpha_6$ be the
highest root, $\theta_2 =
\alpha_1+\alpha_3+\alpha_4+\alpha_5+\alpha_6$, and $\theta_3 =
\alpha_3+\alpha_4+\alpha_5$. Then
$\theta_1^\vee+\theta_2^\vee+\theta_3^\vee=-\varpi_2^\vee+\varpi_3^\vee+\varpi_5^\vee=s_{\alpha_2}s_{\alpha_4}(\varpi_4^\vee)$.
Thus, the considered nilpotent orbit contains the element
$e_{\theta_1'}+e_{\theta_2'}+e_{\theta_3'}$ for
$\theta_1',\theta_2',\theta_3'$ defined by $\theta_i' =
s_{\alpha_4}s_{\alpha_2}(\theta_i)$. Here, the standard Levi
subgroup $L$ corresponds to the root subsystem of type $A_2\times
A_2\times A_1$ generated by
$\alpha_1,\alpha_2,\alpha_3,\alpha_5,\alpha_6$.

As above, the involution $\sigma$ acts on the roots as the element
$w$ of the Weyl group defined by $w=s_{\theta_3'} \circ
s_{\theta_2'} \circ s_{\theta_1'}
=s_{\alpha_4}s_{\alpha_2}s_{\theta_3}s_{\theta_2}s_{\theta_1}s_{\alpha_2}s_{\alpha_4}$.
By a straightforward computation, we see that
$w(\alpha_1)=-\alpha_6$, $w(\alpha_3)=-\alpha_5$, and
$w(\alpha_2)=\alpha_2$. Thus, $L^\sigma$ is a ``diagonal'' subgroup
of $L$ of type $A_2 \times A_1$.
\end{example}

In particular, in both examples, we see that the pair $(L,L^\sigma)$
fits the combinatorial setting described in Section \ref{cox}.

\section{Focus on nilpotent elements of height 2}

\label{SI.5}

We again consider the decomposition $Z=L_Z\ltimes U_Z$ of
Proposition \ref{P-Levi-Z}. In the previous section, we have shown
that, whenever $e$ belongs to a spherical nilpotent orbit (which is
equivalent to saying that $e$ is a nilpotent element of height at
most 3; see Proposition \ref{P4}), the subgroup $L_Z$ can be
realized as a symmetric subgroup of the Levi subgroup $L$ (possibly
up to certain connected components). In the present section, we
focus on the unipotent subgroup $U_Z$. We point out the following
fact:

\begin{proposition}
\label{P-Z-height2} Assume that $e$ is a nilpotent element of height
2 (i.e., $(\mathrm{ad}\,e)^3=0$, $(\mathrm{ad}\,e)^2\neq 0$). Then
$U_Z=U$, so that $Z=L_Z\ltimes U$ and $Z^0=L_Z^0\ltimes U$ in this
case.
\end{proposition}

\begin{proof}
We consider the grading
$\mathfrak{g}=\bigoplus_{i\in\mathbb{Z}}\mathfrak{g}(i)$ determined
by $h$ (as in Section \ref{SI.1}). From the representation theory of
$\fsl_2$, we know that $\mathrm{ad}\,e$ restricts to a map from
$\fg(i)$ to $\fg(i+2)$ for every $i\in\mathbb{Z}$, moreover
$(\mathrm{ad}\,e)^i$ induces a bijection between $\fg(-i)$ and
$\fg(i)$ for every $i\geq 0$. The assumption that $e$ has height 2
implies that $\mathfrak{g}(i)=0$ whenever $|i|>2$, and we get the
inclusion
\[
\mathrm{Lie}(U)=\bigoplus_{i>0}\fg(i)=\fg(1)\oplus\fgg(2)\subset
\fz_\fgg(e).
\]
Since $U$ is connected, this yields $U\subset Z^0$, hence $U=U_Z$.
\end{proof}

\begin{corollary}
\label{CI.5.1} Assume that $e$ is a nilpotent element of height 2.
Then every $Z$-orbit $\mathcal{Z}$ of the flag variety
$\mathcal{B}=G/B$ contains an element which is fixed by the torus
$\tau(\K^*)$. Moreover, the subset of fixed points
$\mathcal{Z}^\tau$ is a single $L_Z$-orbit.
\end{corollary}

\begin{proof}
This fact follows from the more general assertion claimed in Lemma
\ref{L-tau-fixed-points} below.
\end{proof}

\begin{remark}
\label{R-no-tau-fixedpoint} {\rm (a)} If a $Z$-orbit
$\mathcal{Z}\subset\mathcal{B}$ contains elements fixed by $\tau$,
then we have
\[\lim_{t\to 0}\tau(t)\cdot p\in \mathcal{Z}\quad\mbox{for all $p\in\mathcal{Z}$.}\]
Indeed, writing $p=(\ell u)\cdot p_0$ with $\ell\in L_Z$, $u\in
U_Z$, $p_0\in\mathcal{Z}^\tau$, we get
\[\tau(t)\cdot p=\ell\,(\tau(t)u\tau(t)^{-1})\cdot p_0\to \ell\cdot p_0\in\mathcal{Z}^\tau\quad\mbox{as $t\to 0$.}\]

{\rm (b)} When $e$ has height $3$, the conclusion of Corollary
\ref{CI.5.1} is not valid in general, as shown by the following
example. Let $\Phi^+=\{\epsilon_i\pm\epsilon_j:1\leq i<j\leq
3\}\cup\{\epsilon_i:1\leq i\leq 3\}$ be the usual system of positive
roots of $G=\mathrm{SO}_7(\K)$, and let
$\alpha_1=\epsilon_1-\epsilon_2$, $\alpha_2=\epsilon_2-\epsilon_3$,
$\alpha_3=\epsilon_3$ be the simple roots. Take nonzero root vectors
$e_{\epsilon_1}\in\mathfrak{g}_{\epsilon_1}$ and
$e_{\epsilon_2+\epsilon_3}\in\fg_{\epsilon_2+\epsilon_3}$ and let us
consider the nilpotent element
$e=e_{\epsilon_1}+e_{\epsilon_2+\epsilon_3}$. A cocharacter
associated to $e$ is
$\tau(t)=\epsilon_1^\vee+(\epsilon_2+\epsilon_3)^\vee=\varpi_1^\vee+\varpi_3^\vee$.
Thus $e$ belongs to a spherical nilpotent orbit in view of
Proposition \ref{P.orbits}.

We consider the orthogonal form $\omega$ on $\K^7$ defined on the
standard basis by
\[\omega(\varepsilon_i,\varepsilon_j)=\left\{\begin{array}{ll}1 & \mbox{if $i+j=8$,} \\ 0 & \mbox{otherwise}\end{array}\right.\]
and we realize $G$ as the subgroup of elements $g\in\SL_7(\K)$ which
preserve this bilinear form. The flag variety $\mathcal{B}$ can be
realized as the variety of isotropic flags
\[\{(V_1\subset V_2\subset V_3\subset \K^7):\dim V_i=i,\ V_3\subset V_3^\perp\},\]
where $\perp$ stands for the orthogonal with respect to $\omega$. We
consider the flags $F=(\langle
\varepsilon_1+\varepsilon_2\rangle,\langle
\varepsilon_1+\varepsilon_2,\varepsilon_5\rangle,\langle
\varepsilon_1,\varepsilon_2,\varepsilon_5\rangle)$ and
$F_0=(\langle\varepsilon_2\rangle,\langle\varepsilon_2,\varepsilon_5\rangle,\langle
\varepsilon_1,\varepsilon_2,\varepsilon_5\rangle)$. Note that
$\tau(t)\cdot F=(\langle
t^2\varepsilon_1+t\varepsilon_2\rangle,\langle
t^2\varepsilon_1+t\varepsilon_2,t^{-1}\varepsilon_5\rangle,\langle
t^2\varepsilon_1,t\varepsilon_2,t^{-1}\varepsilon_5\rangle)$ for all
$t\in\K^*$, so that $\lim_{t\to 0}\tau(t)\cdot F=F_0$. However, it
can be computed that $\dim Z\cap\mathrm{Stab}_G(F)=5$ whereas $\dim
Z\cap\mathrm{Stab}_G(F_0)=6$, hence $F$ and $F_0$ belong to distinct
$Z$-orbits (of distinct dimensions). In view of part (a), this
implies that the $Z$-orbit of $F$ contains no element fixed by
$\tau$.
\end{remark}

\newpage

\part{Structure of orbits on the flag variety for the action of certain spherical subgroups obtained through parabolic induction}

\label{part2}

In this part of the paper, we assume that $H\subset G$ is a subgroup
obtained through parabolic induction, i.e.,
\begin{equation}
\label{H} H=M\ltimes U
\end{equation}
where
\begin{itemize}
\item $P\subset G$ is a parabolic subgroup and $P=L\ltimes U$ is a Levi decomposition ($U$ is the unipotent radical, $L$ is a Levi factor of $P$);
\item $M\subset L$ is a spherical subgroup.
\end{itemize}
This is also the situation considered in \cite[Lemma 7]{Brion},
where the weak Bruhat order of the $H$-orbits of the flag variety is
described.

Although it is not assumed in this part, we are mostly interested in
the case where $(L^\sigma)^0\subset M\subset L^\sigma$, for some
involution $\sigma\in\mathrm{Aut}(L)$. For instance, Propositions
\ref{P.symmetric-space} and \ref{P-Z-height2} tell us that the role
of $H$ can be played by the groups $Z$ and $Z^0$ where $Z:=Z_G(e)$
is the stabilizer of a nilpotent element of height 2.

The main goal of this part is to describe the structure of the
orbits of $H$ on the flag variety. In this respect, our main result
is Theorem \ref{T4.1} which shows that each $H$-orbit of
$\mathcal{B}$ has a structure of algebraic affine bundle over an
$M$-orbit of the flag variety $\mathcal{B}_L$ of the Levi subgroup
$L$. This result is shown in Section \ref{section-7}. Before that,
in Section \ref{section-6}, we review some facts on the structure of
parabolic orbits on the flag variety.

Let us mention that the structure of orbits for the action of a
general spherical subgroup on the flag variety has been studied in
\cite[Proposition 2.2]{brion-rational}: in particular, it is shown
that each closed orbit is a flag variety for the spherical subgroup.

\section{Notation and short review on parabolic orbits}

\label{section-6}

We fix a maximal torus $T\subset G$. We also fix a Borel subgroup
$B\subset G$ which contains $T$, and we consider the flag variety
$\mathcal{B}=G/B$. Let $W$ be the Weyl group $N_G(T)/T$ of $G$. The choice of
$B$ determines a set of simple reflections of $W$. Let $\ell(w)$
denote the length of an element $w\in W$ with respect to this set of
simple reflections.

We consider a parabolic subgroup $P\subset G$ and a Levi
decomposition $P=L\ltimes U$. There is no loss of generality in
assuming that $B$ is contained in $P$ and $T$ is contained in $L$.
Then, we can find a cocharacter $\tau:\K^*\to T$ such that $P$, $L$,
$U$ are given by
\begin{equation}
\label{1.1} \left\{
\begin{array}{l}
P=P(\tau):=\{g\in G:\lim\limits_{t\to 0}\tau(t)g\tau(t)^{-1}\mbox{ exists}\},\\
U=U(\tau):=\{g\in G:\lim\limits_{t\to 0}\tau(t)g\tau(t)^{-1}=1_G\}, \\
L=L(\tau):=\{g\in G:\tau(t)g\tau(t)^{-1}=g\ \ \forall
t\in\mathbb{K}^*\}=Z_G(\tau).
\end{array}
\right.
\end{equation}

Since $T\subset B\cap L$, the subgroup $B_L:=B\cap L$ is a Borel
subgroup of $L$.\ We denote $\mathcal{B}_L=L/B_L$.

Let $W_P=N_P(T)/T$. Let $\lW{P}$ be a set of representatives of
minimal length for the quotient $W_P\backslash W$.

 We let $P$ act on $\mathcal{B}$, and
it is well known that there are finitely many orbits for this
action:

\begin{proposition}
\label{P-Bruhat}
\begin{itemize}
\item[\rm (a)]
$\displaystyle \mathcal{B}=\bigsqcup_{w\in \lW{P}}\mathcal{P}_w$
where we denote $\mathcal{P}_w:=PwB/B$;
\item[\rm (b)]
$\dim \mathcal{P}_w=\ell(w)+\dim\mathcal{B}_L$ for all $w\in
\lW{P}$.
\end{itemize}
\end{proposition}

In particular, each $P$-orbit $\mathcal{P}_w$ contains a fixed point
of $T$ (namely $wB/B$), which is a fortiori fixed by the subtorus
$\tau(\K^*)$. The next statement describes the structure of the
orbits, by relying on the fixed point set of $\tau(\K^*)$.

%debut-resume
\begin{proposition}
\label{P1.1} Let $\mathcal{B}^\tau\subset\mathcal{B}$ be the
subvariety of fixed points of $\tau(\K^*)$. Let
$\mathcal{P}\subset\mathcal{B}$ be a $P$-orbit. Let
$\mathcal{P}^\tau=\mathcal{P}\cap\mathcal{B}^\tau$. Take any
$p_0=g_0B\in\mathcal{P}^\tau$. Thus $g_0Bg_0^{-1}$ is a Borel
subgroup of $G$ which contains the torus $\tau(\K^*)$. This
guarantees that $(g_0Bg_0^{-1})\cap L$ is a Borel subgroup of $L$.
\begin{itemize}
\item[\rm (a)]
The map $L\to \mathcal{P}^\tau$, $\ell\mapsto \ell\cdot p_0$ is
surjective and induces an isomorphism of varieties
\[L/(g_0Bg_0^{-1})\cap L\to \mathcal{P}^\tau.\]
In particular $\mathcal{P}^\tau$ is a projective variety, hence
closed in $\mathcal{B}$. It is a connected component of
$\mathcal{B}^\tau$ and every connected component of
$\mathcal{B}^\tau$ is of this form.
\item[\rm (b)] There is a unique algebraic affine bundle $\phi_\mathcal{P}:\mathcal{P}\to\mathcal{P}^\tau$ such that
\[\phi_\mathcal{P}(\ell u\cdot p_0)=\ell\cdot p_0\ \mbox{ for all
$\ell\in L$, all $u\in U$}.\] It does not depend on the choice of
$p_0\in\mathcal{P}^\tau$.
\end{itemize}
\end{proposition}

In particular, part (a) of the Proposition says that the number of
connected components of $\cB^\tau$ is equal to the number of
$\cP$-orbits in $\cB$, which is $\big | W/W_P \big |$.

\begin{proof}
The fact that $(g_0Bg_0^{-1})\cap L$ is a Borel subgroup of $L$ is
Corollary 22.4 in \cite{humphreys}. We first show {\rm (a)}. Let
$p_1\in\mathcal{P}^\tau$. Thus $p_0,p_1$ both belong to
$\mathcal{P}$, hence there is $g\in P$ such that $p_1=g\cdot p_0$.
Moreover, we can write $g=\ell u$ with $\ell\in L$ and $u\in U$.
Using (\ref{1.1}), we have
\[p_1=\tau(t)\cdot p_1=\tau(t)\ell u\cdot p_0=\tau(t)\ell u\cdot (\tau(t)^{-1}\cdot p_0)=\ell(\tau(t)u\tau(t)^{-1})\cdot p_0\stackrel{t\to 0}{\longrightarrow} \ell\cdot p_0\]
hence $p_1=\ell\cdot p_0$. We have shown that the map $L\to
\mathcal{P}^\tau$, $\ell\mapsto \ell\cdot p_0$ is surjective. In
addition, the isotropy group of $p_0(=g_0B)$ (in $L$) is clearly
$L\cap(g_0Bg_0^{-1})$ and the fibers of this map are the cosets of
$L\cap(g_0Bg_0^{-1})$. We deduce that the map $L\to\mathcal{P}^\tau$
induces an isomorphism of varieties $L/(g_0Bg_0^{-1})\cap
L\stackrel{\sim}{\to}\mathcal{P}^\tau$. Since $(g_0Bg_0^{-1})\cap L$
is a Borel subgroup of $L$, we deduce that $\mathcal{P}^\tau$ is a
projective variety. In particular it is closed is $\mathcal{B}$.
Since $L$ is connected, it follows that the subsets
$\mathcal{P}^\tau$ (attached to the various $P$-orbits $\mathcal{P}$
of $\mathcal{B}$) are closed, connected, pairwise disjoint, and they
cover $\mathcal{B}^\tau$; hence they are exactly the connected
components of $\mathcal{B}^\tau$. This shows {\rm (a)}.

Next we show {\rm (b)}. Consider the map
\[\phi:\mathcal{B}\to\mathcal{B}^\tau,\ p\mapsto \lim_{t\to 0}\tau(t)\cdot p.\]
It follows from Bialynicki-Birula's theorem \cite{Bialynicki-Birula}
that the restriction of the map $\phi$ over each connected component
of $\mathcal{B}^\tau$ is an algebraic affine bundle, i.e.,
$\phi|_{\phi^{-1}(\mathcal{P}^\tau)}:\phi^{-1}(\mathcal{P}^\tau)\to
\mathcal{P}^\tau$ is an algebraic affine bundle for every $P$-orbit
$\mathcal{P}$. Let $p\in \mathcal{P}$. We can write $p=\ell u\cdot
p_0$ with $\ell\in L$ and $u\in U$. Using (\ref{1.1}) and the fact
that $p_0$ is fixed by $\tau$, we have
\[\tau(t)\cdot p=\tau(t)\cdot(\ell u\cdot p_0)=\tau(t)\ell u\cdot(\tau(t)^{-1}\cdot p_0)=\ell(\tau(t)u\tau(t)^{-1})\cdot p_0\stackrel{t\to 0}{\longrightarrow}\ell\cdot p_0\]
hence
\[\phi(p)=\phi(\ell u\cdot p_0)=\ell\cdot p_0\in\mathcal{P}^\tau.\]
We get in particular $\phi(\mathcal{P})\subset\mathcal{P}^\tau$, and
this implies $\phi^{-1}(\mathcal{P}^\tau)=\mathcal{P}$. Hence $\phi$
restricts to an algebraic affine bundle
$\phi_\mathcal{P}:\mathcal{P}\to\mathcal{P}^\tau$. Moreover the
previous calculation shows that this map is such that
$\phi_\mathcal{P}(\ell u\cdot p_0)=\ell\cdot p_0$ for all
$(\ell,u)\in L\times U$, whenever $p_0\in \mathcal{P}^\tau$, and
$\phi_\mathcal{P}$ is necessarily unique for satisfying the latter
formula, because every element of $\mathcal{P}$ is of the form $\ell
u\cdot p_0$.
\end{proof}

The following observation will be useful.

\begin{lemma}
\label{L4.0} For every $w\in \lW{P}$, we have $(wBw^{-1})\cap
L=B_L$.
\end{lemma}

\begin{proof}
Since $wBw^{-1}$ is a Borel subgroup of $G$ which contains $T$, we
know that $(wBw^{-1})\cap L$ is a Borel subgroup of $L$. To show the
lemma, it suffices to show that $B_L\subset (wBw^{-1})\cap L$.
Assume by contradiction $B_L\not\subset(wBw^{-1})\cap L$. Since both
subgroups are Borel subgroups of $L$ which contain the maximal torus
$T$, it follows that there is a simple root $\alpha$ in the root
system of $L$ such that $\mathfrak{g}_\alpha$ is not contained in
$\mathrm{Lie}(wBw^{-1})$. Hence $w^{-1}(\alpha)<0$. This implies
$\ell(w^{-1}s_\alpha)=\ell(w^{-1})-1$, hence $\ell(s_\alpha
w)<\ell(w)$. Since $s_\alpha\in W_P$, this contradicts the
assumption that $w$ is of minimal length among its coset $W_Pw$.
\end{proof}

Finally, the following statement summarizes the conclusions of
Propositions \ref{P-Bruhat}--\ref{P1.1}; it also uses Lemma
\ref{L4.0}.

\begin{proposition}
\label{P4.2} Let $w\in \lW{P}$. We denote
$\mathcal{P}_w^\tau=\mathcal{P}_w\cap\mathcal{B}^\tau$. Let
$\phi_w=\phi_{\mathcal{P}_w}:\mathcal{P}_w\to\mathcal{P}_w^\tau$ be
the map defined in Proposition \ref{P1.1}.
\begin{itemize}
\item[\rm (a)]
We have $\mathcal{P}_w^\tau=LwB/B$. The map $L/B_L\to
\mathcal{P}_w^\tau$, $\ell B_L\mapsto \ell wB/B$ is an isomorphism
of algebraic varieties.
\item[\rm (b)]
The map $\phi_w:\mathcal{P}_w\to\mathcal{P}_w^\tau$ is an algebraic
affine bundle whose typical fiber is the affine space of dimension
$\ell(w)$.
\end{itemize}
\end{proposition}

\section{Parametrization and structure of $H$-orbits}

\label{section-7}

We start with a preliminary observation.

\begin{lemma}
\label{L-tau-fixed-points} Let $H=M\ltimes U$ be as in (\ref{H}).
Then, every $H$-orbit $\mathcal{H}$ of the flag variety
$\mathcal{B}=G/B$ contains an element which is fixed by the torus
$\tau(\K^*)$. Moreover, the subset of fixed points
$\mathcal{H}^\tau$ is a single $M$-orbit.
\end{lemma}

\begin{proof}
Since $H\subset P$, the orbit $\mathcal{H}$ is contained in some
$P$-orbit $\mathcal{P}$. As recalled in the previous subsection, the
subset of fixed points $\mathcal{P}^\tau$ is nonempty, so let
$p_0\in\mathcal{P}^\tau$. There is $g\in P$ such that $g\cdot
p_0\in\mathcal{H}$. We write $g=u\ell$ with $\ell\in L$ and $u\in
U$. By assumption, $U$ is contained in $H$, so $u\in H$. hence
$\ell\cdot p_0=u^{-1}\cdot(g\cdot p_0)\in\mathcal{H}$. In addition
since $p_0$ is fixed by the torus $\tau(\K^*)$ and $\ell(\in L)$
commutes with every $\tau(t)$ ($t\in\K^*$), we get that $\ell\cdot
p_0$ is also fixed by $\tau(\K^*)$.\ Hence
$\mathcal{H}^\tau\not=\emptyset$.

Now let $p_0,p_1\in\mathcal{H}^\tau$. There is $g\in H$ such that
$p_1=g\cdot p_0$. Writing $g=\ell u$ with $\ell\in M$ and $u\in U$,
we have for every $t\in\K^*$ that
\[p_1=\tau(t)\cdot p_1=\ell(\tau(t)u\tau(t)^{-1})\cdot p_0,\]
hence, by passing to the limit as $t\to 0$, we obtain $p_1=\ell\cdot
p_0$.
\end{proof}

By assumption $M$ is a spherical subgroup of $L$, hence the action
of $M$ on the flag variety $\mathcal{B}_L=L/B_L$ has finitely many
orbits. Let $\Xi\subset L$ be a set of representatives of the
orbits, i.e., giving rise to the decomposition
\[\mathcal{B}_L=\bigsqcup_{\xi\in \Xi}\mathcal{M}_L(\xi)\quad\mbox{where}\quad \mathcal{M}_L(\xi):=M\xi B_L/B_L.\]
Let $d(\xi)=\dim\mathcal{M}_L(\xi)$.

\begin{theorem}
\label{T4.1} For every $(w,\xi)\in \lW{P}\times \Xi$, we set
$\mathcal{H}_{w,\xi}:= H\xi wB/B$. (This is an $H$-orbit of
$\mathcal{B}$.)
\begin{itemize}
\item[\rm (a)] Every $H$-orbit of $\mathcal{B}$ is of the form
$\mathcal{H}_{w,\xi}$ for a unique pair $(w,\xi)\in \lW{P}\times
\Xi$.
\item[\rm (b)] The $\tau$-fixed point set $\mathcal{H}_{w,\xi}^\tau$
is given by $\mathcal{H}_{w,\xi}^\tau=M\xi wB/B$ and it is
($M$-equivariantly) isomorphic to $\mathcal{M}_L(\xi)$.
\item[\rm (c)]
The affine bundle $\phi_w:\mathcal{P}_w\to\mathcal{P}_w^\tau$ of
Proposition \ref{P4.2}\,{\rm (b)} satisfies
$\phi_w^{-1}(\mathcal{H}_{w,\xi}^\tau)=\mathcal{H}_{w,\xi}$, hence
it restricts to an algebraic affine bundle
\[\mathcal{H}_{w,\xi}\to\mathcal{H}_{w,\xi}^\tau\cong\mathcal{M}_L(\xi)\]
of fiber isomorphic to $\mathbb{A}^{\ell(w)}$. In particular,
$\dim\mathcal{H}_{w,\xi}=\ell(w)+d(\xi)$.
\end{itemize}
\end{theorem}

\begin{proof}
Since $H\subset P$, we have a partition
\[\mathcal{B}/H=\bigsqcup_{w\in \lW{P}}\mathcal{P}_w/H.\]
Let $w\in \lW{P}$ and let us consider the $P$-orbit $\mathcal{P}_w$
and its set of $H$-orbits. We also consider the fixed point set
$\mathcal{P}_w^\tau$ and its set of $M$-orbits. It follows from
Lemma \ref{L-tau-fixed-points} that the map
\[\mathcal{P}_w/H\to \mathcal{P}_w^\tau/M,\ \mathcal{H}\mapsto\mathcal{H}^\tau\]
is well defined, in addition this map is clearly injective and
surjective, hence bijective. Let us analyze the decomposition of the
subvariety $\mathcal{P}_w^\tau$ into $M$-orbits. As noted in
Proposition \ref{P4.2}, we have an isomorphism
\[\mathcal{B}_L=L/B_L\to \mathcal{P}_w^\tau,\ \ell B_L\mapsto \ell w B/B.\]
Therefore, the decomposition $\mathcal{B}_L=\bigsqcup_{\xi\in
\Xi}\mathcal{M}_L(\xi)$, with $\mathcal{M}_L(\xi)=M\xi B_L/B_L$,
induces the decomposition into $M$-orbits
\[
\mathcal{P}_w^\tau=\bigsqcup_{\xi\in\Xi} M\xi w B/B,\quad
\mbox{and}\quad \mathcal{M}_L(\xi)\cong M\xi w B/B.
\]
The definition of $\mathcal{H}_{w,\xi}$ implies that
$\mathcal{H}_{w,\xi}$ is the unique $H$-orbit of $\mathcal{P}_w$
such that $\mathcal{H}_{w,\xi}^\tau=M\xi w B/B$. This shows parts
{\rm (a)} and {\rm (b)}.

In order to show part {\rm (c)}, it suffices to show the inclusion
$\phi_w(\mathcal{H})\subset \mathcal{H}^\tau$ for every $H$-orbit
$\mathcal{H}\subset\mathcal{P}_w$. Then, knowing that the $H$-orbits
are pairwise disjoint, this forces
$\phi_w^{-1}(\mathcal{H}_{w,\xi}^\tau)=\mathcal{H}_{w,\xi}$, and the
properties of the restriction $\phi_w|_{\mathcal{H}_{w,\xi}}$ are
inherited from the properties of $\phi_w$ stated in Proposition
\ref{P4.2}\,{\rm (b)}. So let $\mathcal{H}\subset\mathcal{P}_w$ be
an $H$-orbit. Take any $p_0\in \mathcal{H}^\tau$. Thus
$\mathcal{H}=\{\ell u\cdot p_0:\ell\in M,\ u\in U\}$, and by the
characterization of $\phi_w=\phi_{\mathcal{P}_w}$ given in
Proposition \ref{P1.1}, we have $\phi_w(\ell u\cdot p_0)=\ell\cdot
p_0\in\mathcal{H}^\tau$ for all $\ell u\cdot p_0\in\mathcal{H}$.
This shows part {\rm (c)} of the statement.
\end{proof}

We summarize the obtained results about $H$-orbits with the
following diagram:
\begin{equation}
\label{equa:diagram} \xymatrix{
\mathbb{A}^{\ell(w)} \ar[r] & \ar[r]^{\phi_w} \cP_w & \cP_w^\tau \ar[r]^\sim & L/B_L \\
\ar@{=}[u] \mathbb{A}^{\ell(w)} \ar[r] & \ar[r]^{\phi_w} \cH_{w,\xi}
\ar@{^{(}->}[u] & \ar[r]^\sim \cH_{w,\xi}^\tau \ar@{^{(}->}[u] &
\cM_L(\xi) \ar@{^{(}->}[u] }
\end{equation}

\begin{example}
\label{E7.3} Let $G=\mathrm{GL}_n(\K)$ and let $H=Z:=Z_G(e)$ be the
stabilizer of a nilpotent matrix $e\in\mathfrak{gl}_n(\K)$ such that
$e^2=0$. Let $r=\mathrm{rank}\,e$. As shown in Example \ref{E4.5}
and Proposition \ref{P-Z-height2}, the parabolic subgroup
$P=L\ltimes U$ associated to $e$  has a Levi factor of the form
$L=\mathrm{GL}_{r}(\K)\times \mathrm{GL}_{r}(\K)\times
\mathrm{GL}_{n-2r}(\K)$, while $Z$ is of the form $L_Z\ltimes U$
with $L_Z=\Delta\mathrm{GL}_r(\K)\times\mathrm{GL}_{n-2r}(\K)$.
(Here $\Delta\mathrm{GL}_r(\K)$ stands for the diagonal embedding of
$\mathrm{GL}_r(\K)$ into
$\mathrm{GL}_r(\K)\times\mathrm{GL}_r(\K)$.)

In this example:
\begin{itemize}
\item $\lW{P}$ is the set of minimal length representatives of the quotient $(\mathfrak{S}_r\times\mathfrak{S}_r\times\mathfrak{S}_{n-2r})\backslash \mathfrak{S}_n$;
\item
 the flag variety $\mathcal{B}_L$ is a triple flag variety isomorphic to $\mathcal{B}_r\times\mathcal{B}_r\times\mathcal{B}_{n-2r}$, where $\mathcal{B}_k$ stands for the flag variety $\mathrm{GL}_k(\K)/B_k$ of $\mathrm{GL}_k(\K)$;
\item
 the $L_Z$-orbits of $\mathcal{B}_L$ are parametrized by the permutations $v\in\mathfrak{S}_r$ and take the form $\mathbb{O}_v\times\mathcal{B}_{n-2r}$,
where
$\mathbb{O}_v=\mathrm{GL}_r(\K)\cdot(vB_r,B_r)\cong\mathrm{GL}_r(\K)\times^{B_r}\mathbb{A}^{\ell(v)}$.
\end{itemize}
Therefore, in this example, the $Z$-orbits of $\mathcal{B}$ are
parametrized by the pairs $(w,v)\in \lW{P}\times\mathfrak{S}_r$, and
each orbit is an algebraic affine bundle over $\mathcal{B}_r\times
\mathcal{B}_{n-2r}$ of fiber isomorphic to the affine space
$\mathbb{A}^{\ell(w)}\times\mathbb{A}^{\ell(v)}$. In particular,
since the double flag variety $\mathcal{B}_r\times
\mathcal{B}_{n-2r}$ has a natural cell decomposition (the product of
the Schubert cell decompositions of $\mathcal{B}_r$ and
$\mathcal{B}_{n-2r}$), we deduce that each $Z$-orbit of
$\mathcal{B}$ has a cell decomposition; moreover, the number of
cells and the codimensions of the cells are the same for each
$Z$-orbit.
\end{example}

\newpage

\part{Bruhat order}

\label{part3}

In this part, our motivation is to understand the (strong) Bruhat
order for the orbits of $Z=Z_G(e)$ on the flag variety
$\mathcal{B}=G/B$, or equivalently for the orbits of $B$ on the
nilpotent orbit $G\cdot e=G/Z$, where $e$ is nilpotent element of
height $2$. In type $A$, the order is described by Boos and
Reineke~\cite{BR} in terms of link patterns. Their approach is based
on representations of quivers.

Actually, our motivation is (more generally) to have a description
of the (strong) Bruhat order for the orbits of a subgroup $H$ of the
form
\begin{equation}
\label{H-part3} H=M\ltimes U
\end{equation}
where $P=LU$ is the Levi decomposition a parabolic subgroup and
$(L^\sigma)^0\subset M\subset L^\sigma$ for some involution
$\sigma\in \mathrm{Aut}(L)$.

Note that if $\sigma=\mathrm{id}_L$, then $H$ is a parabolic
subgroup and the strong order is in this case the Bruhat order on
the parabolic quotient $\lW{P}$ of the Weyl group. If $P=G$, then
$H$ is a symmetric subgroup of $G$ (possibly up to certain connected
components) and the strong order is described by Richardson and
Springer \cite{RS1,RS2}. If $e$ is a nilpotent element of height 2,
then its stabilizer $Z$ (as well as the connected subgroup $Z^0$)
are of the form considered in (\ref{H-part3}) (see Propositions
\ref{P.symmetric-space} and \ref{P-Z-height2}). Note also that the
description of the weak Bruhat order for $H$ as in (\ref{H-part3})
(or even more generally for $H$ obtained through parabolic induction
from a spherical subgroup of $L$) is given in \cite{Brion}.

In this part, we propose a combinatorial model
which reflects the geometric situation of (\ref{H-part3}). We first
introduce a Coxeter-theoretic partial order, defined by using a
parabolic subgroup of a Coxeter group $W$ which is equipped with an
involution (Section \ref{cox}). In type $A$, this order coincides
with the order given by inclusions of the $Z_G(e)$-orbit closures
(for a spherical nilpotent orbit $G\cdot e$). We give
additional combinatorial criteria in this case (Section \ref{type_a}). We also give
necessary and sufficient conditions for a relation to be a cover
relation (see Theorem \ref{covers}); in~\cite{BR}, necessary
conditions for a relation to be a cover relation are given, but
these conditions are not sufficient in general.

The notation used in Part \ref{part3} is independent of the notation
used in Parts \ref{part1}--\ref{part2}.

\section{A Coxeter-theoretic partial order}

\label{cox}

Let $(W,S)$ be a Coxeter system (with $S$ finite, but $W$ may be
infinite). We denote by $\leq$ the (strong) Bruhat order on $W$.
Recall that for $u,v\in W$, we have $u\leq v$ if and only if every
reduced expression of $v$ has a subword which is a reduced
expression of $u$ (see~\cite{Deodhar}). Also recall that every
subset $L\subset S$ induces a decomposition
$$
W=W^LW_L\cong W^L\times W_L
$$
where $W_L\subset W$ is the subgroup generated by $L$ and
$W^L:=\{w\in W: \ell(ws)>\ell(w)~\forall s\in L\}$.

Let $I, J, K\subset S$ satisfying the following two conditions:

\begin{enumerate}
\item There is an isomorphism $\sigma: W_I \longrightarrow W_J$, $x\mapsto x^*$ of Coxeter groups,
\item $W_{I\cup J \cup K} = W_I \times W_J \times W_K$.
\end{enumerate}

It follows in particular from the first point that $x\mapsto x^*$
preserves the Bruhat order, that is, we have $x\leq y$ if and only
if $x^*\leq y^*$. The second point means that the subsets $I,J,K$
are disjoint and disconnected. Note that $W_I \to W, x \mapsto xx^*$
is  a monomorphism of groups; we denote by $W_{I,J}$ its image:
$$W_{I,J}=\{ xx^*: x\in W_I\}.$$ We have $W_I \cap W_{I,J}=\{1\}$ and
$$W_I\times W_J=W_{I\cup J}= W_I W_{I,J}=W_J W_{I,J}.$$

Using the decomposition $W=W^{I\cup J\cup K} W_{I\cup J\cup K}$, we
get that every $w\in W$ can be written uniquely as a product $x_1
x_2 x_3 x x^*$ with $x_1\in W^{I\cup J\cup K}$, $x_2, x\in W_I$,
$x_3\in W_K$. Note that $$\ell(w)=\ell(x_1)+\ell(x_2 x
x^*)+\ell(x_3)= \ell(x_1)+\ell(x_2x)+\ell(x)+\ell(x_3).$$

We set $W(I,J,K)=W^{I\cup J\cup K} W_I=W^{J\cup K}$. Note that it
gives a set of representatives of the left cosets of the subgroup
$W_K W_{I,J}\subset W$.
\begin{pe}
phrase suivante ajoutée
\end{pe}
The following order defined on this quotient will in fact not depend on this particular
choice of representatives in view of Corollary \ref{coro:no_matter_which_minimal_above}.

Given $w\in W$, we set $[w]=w W_K W_{I,J}$.
In this section, our aim is to study the left cosets $[w]$, the set
of representatives $W(I,J,K)$, and the following binary relation:
$$
\mbox{For $w, w'\in W(I,J,K)$, we write $w' \leq_{\mathcal{O}} w$ if
there is $u\in [w']$ such that $u \leq w$.}
$$
Our first claim is:

\begin{proposition}
\label{P-order_abstract} The relation $\leq_{\mathcal{O}}$ defines a
partial order on $W(I,J,K)$.
\end{proposition}

This assertion is shown in Section \ref{section-partial-order},
where preliminary remarks are also made. In Section
\ref{section-minimal-length}, we describe the elements of minimal
length $\ell$ in each left coset: it turns out that there can be several elements of minimal length in a coset. We show that the definition of
the order $\leq_\mathcal{O}$ can be rephrased in terms of these
minimal elements (Corollary
\ref{coro:no_matter_which_minimal_above}), so that
$\leq_\mathcal{O}$ induces an order on the quotient set
$W/W_{I,J}W_K$ in a way which does not depend on the choice of the
set of representatives $W(I,J,K)$. In Section
\ref{section-cover-relations}, we characterize the cover relations
for $\leq_\mathcal{O}$, and our final conclusion is that the poset
$(W(I,J,K), \leq_{\mathcal{O}})$ is graded with rank function given
by the restriction of the length $\ell$ to $W(I,J,K)$ (Corollary
\ref{C-graded}).

\subsection{Partial order $\leq_\mathcal{O}$}
The proof of Proposition \ref{P-order_abstract} is based on the
following lemma.

\label{section-partial-order}

\begin{lemma}
\label{repn} Let $w, w_1\in W(I,J,K)$ and assume that
$w\leq_{\mathcal{O}} w_1$. Then for every $w_2\in [w_1]$, there is
$v\in [w]$ such that $v \leq w_2$.
\end{lemma}

To prove this lemma, we need:

\begin{lemma}\label{lem_product}
Let $x,y\in W$. Let $x'\leq x$. There is $y'\leq y$ such that
$x'y'\leq xy$.
\end{lemma}

\begin{proof}[Proof of Lemma \ref{lem_product}]
We argue by induction on $\ell(y)$. If $\ell(y)=0$, then the claim
holds with $y'=\eW$. If $\ell(y)=1$,
then $y=s\in S$. If $xs > x$ then the claim holds with $y'=\eW$.
Hence assume that $xs<x$. We choose a reduced expression $s_1
s_2\cdots s_k$ of $x$ with $s=s_k$. Since $x'\leq x$ we have that
$x'$ has a reduced expression occuring as a subword of $s_1 s_2
\cdots s_k$. If $s_k$ does not contribute to that subword, then
$x'\leq xs=xy$ and we are done with $y'=\eW$. Hence assume that
$s_k=s$ contributes to the subword giving a reduced expression of
$x'$. Deleting the last letter gives a subword of $s_1 s_2\cdots
s_{k-1}=xs=xy$, hence with $y'=s$ we get the claim.

Now assume that $\ell(y)>1$. Let $s\in S$ be such that $ys < y$. By
induction on length, there is $u\leq ys$ such that $x'u\leq xys$. By
the case of length one, there is $v\leq s$ such that $x'uv \leq xy$.
Since $ys<y$, we have that $y':=uv$ has a (not necessarily reduced)
expression which occurs as a subword of a reduced expression of $y$.
Since every expression for an element of a Coxeter group has a
reduced expression for it as a subword, we have $y'\leq y$, and we
have already seen that $x'y'\leq xy$.
\end{proof}

\begin{proof}[Proof of Lemma \ref{repn}]
By assumption, there is $u\in [w]$ such that $u\leq w_1$. Write
$w_2=w_1 axx^*$ with $x\in W_I$, $a\in W_K$. It suffices to show
that there are $y\in W_I$, $b\in W_K$ such that $ubyy^*\leq w_1a
xx^*$.

Since $u\leq w_1$, Lemma~\ref{lem_product} yields an element $b\leq
a$ such that $ub\leq w_1a$. By Lemma~\ref{lem_product} again, there
is $y\leq x$ such that $uby\leq w_1a x$. Note that $x\in W_I$ and
$a\in W_K$, hence $y\in W_I$ and $b\in W_K$, because $W_I$ and $W_K$
are parabolic. Now since $w_1 ax\in W(I,J,K) W_K W_I$, we have that
$\ell(w_1 ax x^*)=\ell(w_1ax)+\ell(x^*)$. But $y\leq x$ implies that
$y^*\leq x^*$, hence $ubyy^*$ has a (not necessarily reduced)
expression which is a subword of a reduced expression of $w_1
axx^*$. It follows that $v:=ubyy^*\leq w_1 axx^*=w_2$, which is what
we wanted to show.
\end{proof}

\begin{proof}[Proof of Proposition \ref{P-order_abstract}]
Reflexivity is clear. We show antisymmetry. If $x,y\in W(I,J,K)$
satisfy $x\leq_{\mathcal{O}} y$ and $y\leq_{\mathcal{O}} x$, then
since elements of $W(I,J,K)$ have minimal length in their cosets
modulo $W_K W_{I,J}$, we get that $\ell(x)=\ell(y)$; as there is
$u\in [y]$ such that $u\leq x$, this forces $u=x$, hence $[x]=[y]$,
hence $x=y$.

We now show transitivity. Let $z \leq_{\mathcal{O}} y
\leq_{\mathcal{O}} x$. By definition of the relation
$\leq_\mathcal{O}$, there is $w\in [y]$ such that $w\leq x$. By
Lemma~\ref{repn}, there is $w'\in [z]$ such that $w'\leq w$. We then
have $w'\leq x$ with $w'\in [z]$, which is precisely the definition
of $z \leq_{\mathcal{O}} x$.
\end{proof}

\begin{example}\label{ex:O}
In type $A_3$ with $I=\{s_1\}$, $J=\{s_3\}$, $K=\emptyset$, the
order $\leq_{\mathcal{O}}$ on $W(I,J,K)$ is given in
Figure~\ref{fig:avant}.
\end{example}

\begin{figure}[htbp]
\begin{center}
\begin{tikzpicture}
\draw (1.9,0.2) -- (1,1.35); \draw (2.1,0.2) -- (3,1.35); \draw
(0.9,1.7) -- (0,2.85); \draw (0,3.2) -- (0,4.35); \draw (0,4.7) --
(0.9,5.85); \draw (1.1,6.2) -- (1.9,7.35); \draw (3.1,1.7) --
(4,2.85); \draw (4,3.2) -- (4,4.35); \draw (4,4.7) -- (3.1,5.85);
\draw (2.9,6.2) -- (2.1,7.35); \draw (1, 1.7) -- (1.9,2.85); \draw
(1.1,1.7) -- (3.9,2.85); \draw (3, 1.7) -- (2.1,2.85); \draw
(2.9,1.7) -- (0.1,2.85); \draw (0.1,3.2) -- (3.9,4.35); \draw
(1.9,3.2) -- (0.1,4.35); \draw (2,3.2) -- (2,4.35); \draw (3.9,3.2)
-- (2.1,4.35); \draw (0.1,4.7) -- (2.9,5.85); \draw (1.9,4.7) --
(1,5.85); \draw (3.9,4.7) -- (1.1,5.85); \draw (2.1,4.7) --
(3,5.85); \draw (2,0) node{\footnotesize $\eW$}; \draw (1,1.5)
node{\footnotesize $s_2$}; \draw (3,1.5) node{\footnotesize $s_1$};
\draw (0,3) node{\footnotesize $s_1 s_2$}; \draw (2,3)
node{\footnotesize $s_3 s_2$}; \draw (4,3) node{\footnotesize $s_2
s_1$}; \draw (0,4.5) node{\footnotesize $s_1 s_3 s_2$}; \draw
(2,4.5) node{\footnotesize $s_3 s_2 s_1$}; \draw (4,4.5)
node{\footnotesize $s_1 s_2 s_1$}; \draw (1,6) node{\footnotesize
$s_2 s_1 s_3 s_2$}; \draw (3,6) node{\footnotesize $s_1 s_3 s_2
s_1$}; \draw (2,7.5) node{\footnotesize $s_2 s_3 s_2 s_1 s_2$};
\end{tikzpicture}
\end{center}
\caption{The order $\leq_{\mathcal{O}}$ on $W(I,J,K)$ in type $A_3$
with $I=\{s_1\}$, $J=\{s_3\}$, $K=\emptyset$.} \label{fig:avant}
\end{figure}
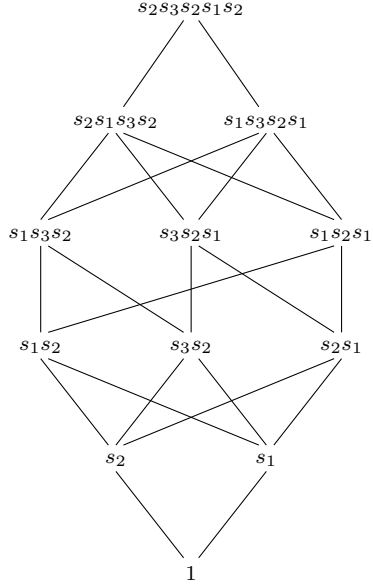

\psfigure{
\begin{figure}[htbp]
\begin{center}
\begin{pspicture}(0,0)(4,7.6)
\psline(1.9,0.2)(1,1.35) \psline(2.1,0.2)(3,1.35)
\psline(0.9,1.7)(0,2.85) \psline(0,3.2)(0,4.35)
\psline(0,4.7)(0.9,5.85) \psline(1.1,6.2)(1.9,7.35)
\psline(3.1,1.7)(4,2.85) \psline(4,3.2)(4,4.35)
\psline(4,4.7)(3.1,5.85) \psline(2.9,6.2)(2.1,7.35) \psline(1,
1.7)(1.9,2.85) \psline(1.1,1.7)(3.9,2.85) \psline(3, 1.7)(2.1,2.85)
\psline(2.9,1.7)(0.1,2.85) \psline(0.1,3.2)(3.9,4.35)
\psline(1.9,3.2)(0.1,4.35) \psline(2,3.2)(2,4.35)
\psline(3.9,3.2)(2.1,4.35) \psline(0.1,4.7)(2.9,5.85)
\psline(1.9,4.7)(1,5.85) \psline(3.9,4.7)(1.1,5.85)
\psline(2.1,4.7)(3,5.85)
\rput(2,0){\footnotesize $\eW$} \rput(1,1.5){\footnotesize $s_2$}
\rput(3,1.5){\footnotesize $s_1$} \rput(0,3){\footnotesize $s_1
s_2$} \rput(2,3){\footnotesize $s_3 s_2$} \rput(4,3){\footnotesize
$s_2 s_1$} \rput(0,4.5){\footnotesize $s_1 s_3 s_2$}
\rput(2,4.5){\footnotesize $s_3 s_2 s_1$} \rput(4,4.5){\footnotesize
$s_1 s_2 s_1$} \rput(1,6){\footnotesize $s_2 s_1 s_3 s_2$}
\rput(3,6){\footnotesize $s_1 s_3 s_2 s_1$}
\rput(2,7.5){\footnotesize $s_2 s_3 s_2 s_1 s_2$}
\end{pspicture}
\end{center}
\caption{The order $\leq_{\mathcal{O}}$ on $W(I,J,K)$ in type $A_3$
with $I=\{s_1\}$, $J=\{s_3\}$, $K=\emptyset$.} \label{fig:avant}
\end{figure}
\medskip
}

\begin{remark}\label{rmq:fix}
{\rm (a)} Note that $W(I,J,K)=W^{J\cup K}$, and $W^{J\cup K}$ is
naturally ordered by the restriction of the Bruhat order on $W$. The
orders $\leq$ and $\leq_{\mathcal{O}}$ on $W(I,J,K)$ do \textit{not}
coincide in general. For instance, taking $W, I, J$ and $K$ as in
Example~\ref{ex:O}, we have $s_1\leq_{\mathcal{O}} s_3 s_2$, while
$s_1\not\leq s_3 s_2$.

{\rm (b)} As Lemma~\ref{repn} shows, the partial order $\leq_\cO$ is
also defined on the quotient $W/(W_KW_{I,J}) \simeq W(I,J,K)$ by the
equivalence $[w] \leq_\cO [w']$ if and only if $\forall u' \in [w'],
\exists u \in [w]:u \leq u'$.

{\rm (c)} Let $(W,S)$ be a Coxeter system, $L\subset S$ and $W_L$
the standard parabolic subgroup generated by $L$. Then $(W_L, L)$ is
a Coxeter system. Let $\theta$ be an automorphism of the Dynkin
diagram of $L$; it extends to an automorphism of the Coxeter group
$(W_L, L)$, which we also denote by $\theta$. Let $W_L^\theta$ be
the subgroup of $\theta$-fixed points of $W_L\subset W$. We can
generalize the partial order $\leq_{\mathcal{O}}$ to $W/W_L^\theta$
as follows: for $u, v\in W$, we have $uW_L^\theta \leq_{\mathcal{O}}
v W_L^\theta$ if and only if for every $w\in v W_L^\theta$, there is
$w'\in u W_L^\theta$ such that $w'\leq w$. In the case described
above, we have $L=I\cup K \cup J$, and the automorphism $\theta$ of
$W_L=W_I \times W_K \times W_J$ is given by $(x,y,z^*)\mapsto
(z,y,x^*)$. We have $W_L^\theta = W_K W_{I,J}$. The fact that the
two partial orders coincide follows from Lemma~\ref{repn} (or from
part {\rm (b)} of this remark).
\end{remark}

\subsection{Elements of minimal length in $[w]$}

\label{section-minimal-length} Given $w\in W(I,J,K)$, the left coset
$[w]$ does not have a single representative of minimal length in
general. But $w$ has minimal length in $[w]$: indeed, writing $w=
w_1 w_2$ with $w_1\in W^{I\cup J\cup K}$, $w_2\in W_I$, then for any
$x\in W_I$, $a\in W_K$ and $w'=w_1 w_2 a xx^*$ we have that
$$\ell(w')=\ell(w_1)+\ell(w_2x)+\ell(x^*)+\ell(a)=\ell(w_1)+\ell(w_2
x)+\ell(x)+\ell(a),$$ and since $\ell(w_2 x)+\ell(x) \geq \ell(w_2)$
we get that $\ell(w')\geq\ell(w)$. Moreover, every element of
minimal length in $[w]$ lies in $W^K$, but the converse does not
hold in general.

For $w\in W(I,J,K)$, let $$\mathrm{Min}(w)=\{ u\in [w]:
\ell(u)=\ell(w)\}$$ be the set of elements of minimal length in the
coset $[w]$. On the basis of the above calculation, we have the
following easy results:

\begin{lemma}
\label{lemma:min} Let $w=w_1w_2 \in W(I,J,K)$ with $w_1\in W^{I\cup
J\cup K}$ and $w_2\in W_I$.
\begin{itemize}
 \item[\rm (a)] $\displaystyle \mathrm{Min}(w)=\{ wxx^{*}: x\in W_I,\ \ell(w_2x)+\ell(x^{-1})=\ell(w_2)\}$;
 \item[\rm (b)] $\displaystyle \mathrm{Min}(w)=w_1 \cdot \mathrm{Min}(w_2)$.
\end{itemize}
\end{lemma}

In particular, the set $\mathrm{Min}(w)$ is in bijection with the
set of elements lying below $w_2$ for the right weak order, and we
have $|\mathrm{Min}(w)|=1$ iff $w_2=\eW$ iff $w\in W^{I\cup J\cup
K}$.

\begin{lemma}
\label{minimal_below} Let $w\in W(I,J,K)$, $v\in [w]$. There is
$u\in \mathrm{Min}(w)$ such that $u\leq v$. In other words,
$\mathrm{Min}(w)$ is also characterized as being the set of elements
in $[w]$ which are minimal for the Bruhat order.
\end{lemma}

\begin{proof}
Let $w=w_1 w_2$ as above. Let $v=w_1 w_2 a xx^*$, with $a\in W_K$,
$x\in W_I$. Note that $v':=w_1 w_2 xx^* \leq v$ since $v=v'a$ and
$\ell(v)=\ell(v')+\ell(a)$. If $v'\in\mathrm{Min}(w)$, then we are
done. Assume that this is not the case: we then claim that there is $y\in
W_I$ such that $w_1 w_2 y y^* < w_1 w_2 x x^*$, which is enough to
conclude, since we can then iterate until reaching an element of
minimal length.

To show the claim, let $s_1 s_2 \cdots s_k$ be a reduced expression
of $w_2x$, and $s_{k+1} \cdots s_m$ be a reduced expression of
$x^{-1}$. We have $\ell(w_2xx^*)=\ell(w_2x)+\ell(x^*)=m$ since $I$
and $J$ are disconnected.

Since we assume that $w_1w_2xx^*\notin\mathrm{Min}(w)$, we have
$\ell(w_2 x)+\ell(x^{-1})>\ell(w_2)$ (by Lemma \ref{lemma:min}),
hence the expression $s_1 s_2 \cdots s_m$ is not reduced. Since both
$s_1 s_2\cdots s_k$ and $s_{k+1}\cdots s_m$ are reduced, there
exists a minimal $i\in\{k+1, \dots, m\}$ such that $s_1 s_2 \cdots
s_i$ is not reduced, and by the exchange lemma we have $s_1 s_2
\cdots s_i$= $s_1 s_2 \cdots \widehat{s_j} \cdots s_{i-1}$ for some
$j\in \{1, \dots, k\}$, where the notation $\widehat{s_j}$ indicates
that $s_j$ is removed. It follows that
$$z':=w_1 w_2 x x^* (s_{k+1}\cdots s_i) (s_{k+1}^*\cdots s_i^*)=w_1 (s_1 \cdots \widehat{s_j} \cdots s_{i-1})
(s_m^* s_{m-1}^*\cdots s_{i+1}^*).$$ Multiplying $z'$ on the right
by $(s_{i-1} s_{i-2} \cdots s_{k+1})(s_{i-1}^* s_{i-2}^* \cdots
s_{k+1}^*)$ we get an element
$$z:=w_1 (s_1 \cdots \widehat{s_j} \cdots s_{k}) (s_m^* s_{m-1}^*\cdots \widehat{s_{i}^*} \cdots s_{k+1}^*)\in [w].$$
We have $(s_1 \cdots \widehat{s_j} \cdots s_{k}) < w_2x= s_1 \cdots
s_k$ and $(s_m^* s_{m-1}^*\cdots \widehat{s_{i}^*} \cdots s_{k+1}^*)
< x^*= s_m^*\cdots s_{k+1}^*$, hence since $w_1 \in W^{I\cup J\cup
K}$, $w_2x\in W_I$ and $x^*\in W_J$, we get $z  < w_1 w_2 x x^*$.
Since $z= w_1 w_2 y y^*$ where $y=x (s_{k+1} s_{k+2}\cdots s_i
s_{i-1} \cdots s_{k+1})$, the proof is complete.
\end{proof}

If $\theta$ is an automorphism of a Coxeter group $(W,S)$ induced by
an automorphism of the Dynkin diagram, and if each orbit of $\theta$
in $S$ generates a finite group, then it is well known that the
subgroup of fixed points $W^\theta$ is again a Coxeter group (see
for instance~\cite{Muhlh}). Moreover, the canonical generators of
$W^\theta$ are obtained as follows. For each orbit of $\theta$ on
the simple system $S$, consider the standard parabolic subgroup
generated by the simple reflections in this orbit. Take the longest
element in this subgroup (which is finite by assumption). Then the
set of all such longest elements, for all orbits, forms the simple
system of the subgroup of fixed points.

In our case, considering the automorphism $\theta$ of $L=I\cup J\cup
K$ as in Remark~\ref{rmq:fix}(c), we have $W_L^\theta=W_K W_{I,J}$;
the orbits of $\theta$ are of the form $\{s\}$ for $s\in K$ or
$\{s,s^*\}$ for $s\in I$. Hence the corresponding generators of $W_K
W_{I,J}$ viewed as Coxeter group are given by $\{s: s\in K\}\cup\{
ss^*: s\in I\}$. We write $\Theta=\{ ss^*: s\in I\}$.

\begin{lemma}\label{lem_min}
Let $u,u'\in \mathrm{Min}(w)$. There is a sequence $$u,\ ux_1,\ ux_1
x_2,\ \dots,\ ux_1 x_2\cdots x_k=u'$$ with $x_1, x_2, \ldots,
x_k\in\Theta$ such that for all $i$, $u x_1 x_2\cdots x_i\in
\mathrm{Min}(w)$. In other words, any two elements in
$\mathrm{Min}(w)$ can be related by multiplying on the right by a
sequence of generators of $W_K W_{I,J}$ coming only from $\Theta$,
and such that at each step, the obtained element still has minimal
length in $[w]$.
\end{lemma}

\begin{proof}
Since the elements of $\Theta$ have order two, it suffices to show
the claim for $u'=w$. Let $w=w_1 w_2$ as above. By Lemma
\ref{lemma:min}, there is $x\in W_I$ such that
$\ell(w_2x)+\ell(x^{-1})=\ell(w_2)$ and $u=w xx^*=w_1 w_2x x^*$. Let
$x=s_k s_{k-1} \cdots s_2 s_1$ be a reduced expression of $x$.
Setting $x_i=s_i s_i^*$ and using the fact that $I$ and $J$ are
disconnected and that $\ell(w_2x)+\ell(x^{-1})=\ell(w_2)$ (hence
that $w_2$ has a reduced expression ending by $x^{-1}$), we obtain
$ux_1x_2\cdots x_i= w_1 w_2 s_k s_{k-1} \cdots s_{i+1} s_k^*
s_{k-1}^* \cdots s_{i+1}^*$ and
\begin{eqnarray*}
\ell(u x_1x_2\cdots x_i)&=&\ell(w_1)+ \ell(w_2 s_k s_{k-1} \cdots s_{i+1}) +\ell(s_k^* s_{k-1}^* \cdots s_{i+1}^*)\\
&=&\ell(w_1)+\ell(w_2)-(k-i) +(k-i)=\ell(w_1 w_2)=\ell(w),
\end{eqnarray*}
which concludes the proof.
\end{proof}

\begin{lemma}\label{lemm:couverture}
Let $w \in W(I,J,K)$. Let $u \in \mathrm{Min}(w)$, $u' \in W$ and $s
\in I$. Assume that $u' < u$ and $\ell(uss^*)=\ell(u)$. Then there
exists $y$ in $\{u',u'ss^*\}\subset [u']$ such that $\ell(y)\leq
\ell(u')$ and $y < uss^*$. In particular, if $u'\in\mathrm{Min}(w')$
for some $w'\in W(I,J,K)$, then $y\in\mathrm{Min}(w')$.
\end{lemma}

\begin{proof}
First, assume that $\ell(us)=\ell(u)-1$. Let $u=s_1 s_2\cdots s_k$
be a reduced expression of $u$ with $s_k=s$. Since $u'< u$, there is
a subword of $s_1 s_2\cdots s_k$ which is a reduced expression for
$u'$, say $s_{i_1} s_{i_2} \cdots s_{i_m}$, $1\leq i_1 < i_2 < \dots
<\ i_m \leq k$. If $k= i_m$, then $u'ss^*=s_{i_1} s_{i_2} \cdots
s_{i_{m-1}}s^*$ which is a subexpression of the expression
$uss^*=s_1 s_2\cdots s_{k-1} s^*$, the latter being reduced since
$\ell(uss^*)=\ell(u)$, hence $y:=u'ss^*$ satisfies $y< uss^*$ and
$\ell(y)\leq \ell(u')$. If $k\neq i_m$, then $u'=s_{i_1}
s_{i_2}\cdots s_{i_m}$ is a subexpression of the reduced expression
$uss^*=s_1 s_2\cdots s_{k-1} s^*$, hence $y:=u'$ satisfies $y<
uss^*$.

Now if $\ell(us)=\ell(u)+1$, then taking a reduced expression
$u=s_1\cdots s_k$, the exchange lemma implies that
$uss^*=s_1\cdots\widehat{s_j}\cdots s_ks$ for some
$j\in\{1,\ldots,k\}$ (since $\ell(uss^*)=\ell(u)$), hence
$us^*=s_1\cdots\widehat{s_j}\cdots s_k$ (since $s$ and $s^*$
commute). We then
have $\ell(us^*)=\ell(u)-1$, and we can argue as above replacing $s$
by $s^*$.
\end{proof}

\begin{corollary}\label{coro:no_matter_which_minimal_above}
Let $w,w'\in W(I,J,K)$. The following conditions are equivalent.
\begin{itemize}
\item[\rm (i)] $w'<_{\mathcal{O}}w$;
\item[\rm (ii)] There are $u\in\mathrm{Min}(w)$ and $u'\in[w']$ such that $u'<u$;
\item[\rm (iii)] For all $u\in\mathrm{Min}(w)$, there is $u'\in[w']$
such that $u'<u$.
\end{itemize}
\end{corollary}

\begin{proof}
The implications (iii) $\Rightarrow$ (i) $\Rightarrow$ (ii)
immediately follow from the definition of the order
$\leq_\mathcal{O}$. For showing (ii)$\Rightarrow$(iii), assume
$u_0\in\mathrm{Min}(w)$ and $u'_0\in[w']$ such that $u'_0<u_0$, and
let $u\in\mathrm{Min}(w)$. By Lemma~\ref{lem_min}, the element $u$
can be reached from $u_0$ by applying a sequence of elements of
$\Theta$ at the right of $u_0$, without increasing the length.
Applying Lemma \ref{lemm:couverture} inductively, we find $u'\in
[w']$ such that $u'< u$.
\end{proof}

Let $$\mathcal{M}=\bigsqcup_{w\in W(I,J,K)} \mathrm{Min}(w).$$ In
the rest of Section \ref{section-minimal-length}, we establish some
properties of the set $\mathcal{M}$. These properties are used in
Section~\ref{section-cover-relations} to study the cover
relations of the order $\leq_\mathcal{O}$.

For $v,w\in W$, we write $v\leq_R w$ if
$\ell(w)=\ell(wv^{-1})+\ell(v)$, that is, if $w$ has a reduced
expression ending with a reduced expression of $v$ (this defines the
so-called \textit{right weak order} on $W$).

Given a subset $L\subset S$, by the parabolic decomposition
$W=W^LW_L$, every $w\in W$ can be written as $w=w^L w_L$ with unique
$w^L\in W^L$ and $w_L\in W_L$. Note that $w_L$ is also characterized
as being the unique element in $W_L$ which is maximal with respect
to $\leq_R$ and such that $w_L\leq_R w$.

Let $T\subset W$ denote the set $\bigcup_{w\in W} wSw^{-1}$, i.e.,
the set of reflections in $W$. For $u, u'\in W$, we write $u'\leqdot
u$ if $u$ covers $u'$ in the (strong) Bruhat order on $W$. Note
that, in such a case, we must have $u'^{-1}u\in T$.

\begin{lemma}\label{lem:cover_right}
Let $L\subset S$. Let $w\in W$, $t\in T$ such that $t\notin W_L$ and
$w \leqdot wt$. Then $(wt)_L \leq_R w_L$.
\end{lemma}

\begin{proof}
We choose a reduced expression $u_1 u_2 \cdots u_\ell s_1 s_2\cdots
s_k$ of $wt$ such that $u_1 \cdots u_\ell$ is a reduced expression
of $(wt)^L$ and $s_1 \cdots s_k$ is a reduced expression of
$(wt)_L$. Since $w\leqdot wt$, we have that $w$ has a reduced
expression which is obtained from $u_1 u_2 \cdots u_\ell s_1
s_2\cdots s_k$ by deleting one letter. If the letter which is
deleted is $s_i$ for some $i$, then $t=s_k s_{k-1} \cdots s_i
s_{i+1} \cdots s_k$, which lies in $W_L$ (because the letters in a
reduced expression of an element in a standard parabolic subgroup
stay in this subgroup), a contradiction. Hence the letter which is
deleted is among the $u_i$'s. The element $w$ has therefore a
reduced expression of the form
$$u_1  \cdots \widehat{u_i} \cdots u_\ell s_1 s_2\cdots s_k,$$
hence it has $s_1 s_2 \cdots s_k=(wt)_L$ as a suffix, implying that
$(wt)_L \leq_R w$. By maximality of $w_L$ with respect to $\leq_R$,
we deduce that $(wt)_L \leq_R w_L$.
\end{proof}

From now on, let $L=I\cup J \cup K$.

\begin{lemma}
\label{lemma:M} We have
\begin{itemize}
 \item[\rm (a)] $\displaystyle \cM = \bigsqcup_{w \in W^L} w \cdot (\cM \cap W_L)$. Hence $w\in\mathcal{M}\Leftrightarrow w_L\in\mathcal{M}$.
 \item[\rm (b)] $\cM\cap W_L = \{ uv^*:u,v\in W_I,\ \ell(uv^{-1})=\ell(u)+\ell(v)\}$.
 \item[\rm (c)] If $W_I$ is finite, then
$$\cM \cap W_L = \{uv^*:u,v \in W_I,\ u \leq_R w_{0,I} v\},$$ where $w_{0,I}$ denotes the longest element in $W_I$.
\end{itemize}
\end{lemma}

\begin{proof}
Part (a) immediately follows from Lemma \ref{lemma:min}(b).

By Lemma \ref{lemma:min}(a), the elements of $\cM\cap W_L$ are
exactly of the form $wxx^*$ with $w \in W_I=W(I,J,K) \cap W_L$ and
$x\in W_I$ such that $\ell(wx)+\ell(x^{-1})=\ell(w)$. Setting $u=wx$
and $v=x$, we get the description given in Part (b).

When $W_I$ is finite, the latter condition on lengths is
%This is
equivalent to $\ell(w_{0,I}w^{-1})+\ell(wx)=\ell(w_{0,I}x)$, or to
the fact that $wx \leq_R w_{0,I}x$. Setting $u=wx$ and $v=x$, we get
the asserted condition $u \leq_R w_{0,I}v$.
\end{proof}

We also have the following results.

\begin{lemma}
\label{lemma:L_parts} Let $L=I\cup J\cup K$ be as above. If
$w\in\mathcal{M}$ and $y\in W$ are such that $y_L \leq_R w_L$, then
$y\in\mathcal{M}$.
\end{lemma}

\begin{proof}
By Lemma \ref{lemma:M}, we have $w_L= uv^*$ with $u,v\in W_I$
such that $\ell(uv^{-1})=\ell(u)+\ell(v)$. If $y_L\leq_R w_L$, then
since $I$ and $J$ are disconnected, we have $y_L=u'v'^*$ with
$u'\leq_R u$, $v'\leq_R v$. Setting $u=u_1u'$, $v=v_1v'$, we have
\begin{eqnarray*}
\ell(w_L)=\ell(u)+\ell(v) & = & \ell(u')+\ell(u_1)+\ell(v')+\ell(v_1) \\
 & \geq & \ell(u_1)+\ell(u'v'^{-1})+\ell(v_1) \\
 & \geq & \ell(u_1u'v'^{-1}v_1^{-1})=\ell(uv^{-1})=\ell(u)+\ell(v).
 \end{eqnarray*}
We deduce that $\ell(u' v'^{-1})=\ell(u')+\ell(v')$, which shows
that $y_L\in \mathcal{M}$ (by Lemma \ref{lemma:M}(b)). Hence $y=y^L
y_L\in\mathcal{M}$ by Lemma \ref{lemma:M}(a).
\end{proof}

\begin{corollary}
Let $w\in\mathcal{M}$. Let $t\in T\setminus W_L$ such that $w\leqdot
wt$. Then $wt\in\mathcal{M}$.
\end{corollary}

\begin{proof}
By Lemma~\ref{lem:cover_right} we have $(wt)_L\leq_R w_L$. By
Lemma~\ref{lemma:L_parts}, we get $wt\in\mathcal{M}$.
\end{proof}

\begin{lemma}\label{equiv_covers}
Let $y\in \mathcal{M}$, $w\in W(I,J,K)$ such that $y\leq w$ and
$\ell(y)+2\leq \ell(w)$. Then there is $v\in \mathcal{M}$ with
$y<v<w$.
\end{lemma}

\begin{proof}
By~\cite[Corollary 3.8]{Deodhar}, the poset $(W^K, \leq)$ is graded
and the rank function is given by the restriction of $\ell$ to
$W^K$. Since $\mathcal{M}\subset W^K$, there is $t\in T$ such that
$y< wt \leqdot w$ and $wt\in W^K$.

Since $w\in W(I,J,K)$, we can write $w=w^Lw_L$ with $w^L\in W^L$ and
$w_L\in W_I$, and $\ell(w)=\ell(w^L)+\ell(w_L)$, hence a reduced
expression of $w$ can be obtained by concatenating reduced
expressions of $w^L$ and $w_L$. Since $wt\leqdot w$, we can get a
reduced expression of $wt$ by deleting either a letter in $w^L$ or a
letter in $w_L$. In the latter case, we have $wt\in
W^LW_I=W(I,J,K)$, hence $wt\in\cM$, and the lemma is obtained with
$v=wt$. It remains to consider the former situation. In such a case,
we have $wt=(w^L t_0)w_L$, where $w^L t_0\leqdot w^L$. Since $w^L\in
W^L$, we must have $t_0\in T\setminus W_L$, and therefore
\begin{equation}
\label{t-notin-WL} t=w_L^{-1} t_0w_L\notin W_L.
\end{equation}

Let $y=y_1 y_2 y_3^*$ be the unique decomposition of $y$ with
$y_1\in W^L$, $y_2,y_3\in W_I$; this follows from Lemma
\ref{lemma:M} (since $y\in\cM$) or more generally from the fact that
$y\in W^K$. Since $wt\in W^K$, we can decompose $wt=w_1 w_2 w_3^*$
in the same way. Since $y\leq wt$, there are $\hat{w}_1$,
$\hat{w}_2$, $\hat{w}_3$ such that $y=\hat{w}_1 \hat{w}_2
\hat{w}_3^*$, $\hat{w}_i\leq w_i$ and
$\ell(y)=\ell(\hat{w}_1)+\ell(\hat{w}_2)+\ell(\hat{w}_3)$: just take
a reduced expression of $wt$ which is obtained by concatenating
reduced expressions of $w_1$, $w_2$ and $w_3^*$; it then has a
subword which is a reduced expression of $y$. Note that since
$\hat{w}_2 \hat{w}_3^*\in W_{I\cup J}\subset W_L$, we have that
$\ell(w_1 \hat{w}_2
\hat{w}_3^*)=\ell(w_1)+\ell(\hat{w}_2)+\ell(\hat{w}_3)$. Hence
$y\leq w_1 \hat{w}_2 \hat{w}_3^* \leq wt$.

\smallskip
\noindent {\it Claim 1:} $w_1 \hat{w}_2 \hat{w}_3^*\in \mathcal{M}$.

To see this, note that since
$\ell(y)=\ell(\hat{w}_1)+\ell(\hat{w}_2\hat{w}_3^*)$, we have
$\hat{w}_2\hat{w}_3^*\leq_R y$, and so $\hat{w}_2 \hat{w}_3^*\leq_R
y_L=y_2y_3^*$ (since $\hat{w}_2 \hat{w}_3^*\in W_L$). Thus
$y_2=u_2\hat{w}_2$, $y_3= u_3\hat{w}_3$ for some $u_2,u_3\in W_I$
such that
$\ell(y_i)=\ell(u_i)+\ell(\hat{w}_i)$ (using that $I$ and $J$ are disconnected). But
since $y\in\mathcal{M}$, by Lemma \ref{lemma:M} we have $\ell(y_2 y_3^{-1})=\ell(y_2)+\ell(y_3)$ and we get
\begin{eqnarray*}
\ell(u_2)+\ell(\hat{w}_2\hat{w}_3^{-1})+\ell(u_3) & \geq &
\ell(u_2\hat{w}_2 \hat{w}_3^{-1} u_3^{-1})=\ell(y_2 y_3^{-1})=\ell(y_2)+\ell(y_3) \\
 & = & \ell(u_2)+\ell(\hat{w}_2)+\ell(\hat{w}_3)+\ell(u_3) \\
 & \geq & \ell(u_2)+\ell(\hat{w}_2\hat{w}_3^{-1})+\ell(u_3),
\end{eqnarray*}
hence $\ell(\hat{w}_2
\hat{w}_3^{-1})=\ell(\hat{w}_2)+\ell(\hat{w}_3)$, implying Claim 1
(see Lemma \ref{lemma:M}).

\smallskip

Since $y\leq w_1 \hat{w}_2 \hat{w}_3^* < w$, if $y\neq w_1 \hat{w}_2
\hat{w}_3^*$, then we get the conclusion of the lemma with $v=w_1
\hat{w}_2 \hat{w}_3^*$. Hence we can assume that $y=w_1 \hat{w}_2
\hat{w}_3^*$.

\smallskip
\noindent {\it Claim 2:} We can find $t'\in T$, $t'\notin W_L$, such
that $y\leqdot yt' < w$.

Since $y$ has a reduced expression obtained from $w_1 w_2 w_3^*=wt$
by only deleting letters in $w_2$ and $w_3^*$, it follows that $wt=
y t_1 t_2\cdots t_i$ with $t_j\in W_{I\cup J}\cap T\subset W_L$ for
all $j=1, \dots, i$ and $y t_1 t_2 \cdots t_j \leqdot y t_1 t_2
\cdots t_{j+1}$ for all $j=1, \dots, i-1$. Hence we have $$y t_1
t_2\cdots t_{i-1} \leqdot y t_1 t_2\cdots t_i=wt \leqdot w.$$
Setting $u= y t_1 t_2\cdots t_{i-1}$, we have $u \leqdot ut_i= wt
\leqdot w=ut_i t$.
By a property of Bruhat intervals (see~\cite[Proposition 2.1 and its
proof]{Dyer}), the Bruhat interval $[u,w]$ is isomorphic (as a
poset) to a Bruhat interval in a dihedral reflection subgroup of
$W$. Hence there is exactly one element $u'\in W$ such that $u
\leqdot u'\leqdot w$, $u'\neq wt$. Let $t_i',q\in T$ be such that
$u'=ut_i'$ and $u' q=w$.

\smallskip
\noindent {\it Subclaim:} $t_i'\notin W_L$.

We have $t_i'q= t_i t\neq 1$ and, by~\cite[Lemma 3.1]{Dyer}, the
reflection subgroup $W':=\langle t_i', q, t_i, t\rangle$ is
dihedral. To show the Subclaim, arguing by contradiction, assume
that $t_i'\in W_L$. We show that this implies that $W'\subset W_L$,
contradicting (\ref{t-notin-WL}).
Note that if $W$ is of type $A$, this is clear as $W'$ has to be either of type $A_1\times A_1$ or of
type $A_2$, and is therefore generated by any two distinct
reflections, whence $W'=\langle t_i, t_i'\rangle \subset W_L$. In
general the result can be proven as follows:
by~\cite[Remark 3.2]{Dyer}, the dihedral reflection subgroup
$\langle t_i', t_i\rangle$ is included in a unique maximal dihedral
reflection subgroup $W''$, defined by
$$
W''=\langle r_{\alpha}~|~\alpha\in (\mathbb{R} \alpha_{t_i'}
+\mathbb{R} \alpha_{t_i})\cap \Phi^+\rangle,
$$
where $\Phi$ denotes a generalized root system for $(W,S)$ and
$\alpha_{t_i}, \alpha_{t_i'}$ are the roots attached to $t_i$ and
$t_i'$. Since $t_i', t_i\in W_L$ which is standard parabolic, it
follows that for every root $\alpha\in (\mathbb{R} \alpha_{t_i'}
+\mathbb{R} \alpha_{t_i})\cap \Phi^+$ we have $r_\alpha\in W_L$,
hence that $W''\subset W_L$. Since $W''$ is the unique maximal
dihedral reflection subgroup containing both $t_i'$ and $t_i$, it
follows that $W'\subset W''\subset W_L$.
The proof of the Subclaim is then complete.

\smallskip

Arguing the same with $y t_1 t_2\cdots t_{i-2} \leqdot yt_1
t_2\cdots t_{i-2} t_{i-1} \leqdot y t_1 t_2\cdots t_{i-2} t_{i-1}
t_i'=u'$, we find $t_{i-1}'\notin W_L$ with $y t_1 t_2\cdots t_{i-2}
\leqdot yt_1 t_2\cdots t_{i-2} t_{i-1}' \leqdot y t_1 t_2\cdots
t_{i-2} t_{i-1} t_i'=u'$. Going on, we find reflections $t_1', t_2',
\cdots, t_i'\notin W_L$ such that $$y \leqdot y t_1' \leqdot y t_1
t_2' \leqdot \dots \leqdot y t_1 t_2\cdots t_{i-1} t_i'=u'\leqdot
w.$$ Hence, taking $t'=t_1'$, we get Claim 2.

\smallskip

Now it suffices to show that $y t'\in \mathcal{M}$. To see this, as
$t'\notin W_L$, by Lemma~\ref{lem:cover_right}, we get that $(y
t')_L \leq_R y_L$. By Lemma~\ref{lemma:L_parts}, it implies that $y
t'\in\mathcal{M}$, as required.
\end{proof}

\subsection{Cover relations for $\leq_{\mathcal{O}}$}
\label{section-cover-relations}

Given $w, w'\in W(I,J,K)$, we write $w'\leqdot_{\mathcal{O}} w$ if
$w$ covers $w'$ in the partial order $\leq_{\mathcal{O}}$. For $u,
u'\in W$, recall that we write $u'\leqdot u$ if $u$ covers $u'$ in
the (strong) Bruhat order on $W$. We now characterize the cover
relations in $\leq_{\mathcal{O}}$ in terms of elements of minimal
length in cosets:

\begin{theorem}\label{covers}
Let $w, w'\in W(I,J,K)$. The following are equivalent:
\begin{itemize}
\item[\rm (i)] We have $w'\leqdot_{\mathcal{O}} w$;
\item[\rm (ii)] There are $u\in \mathrm{Min}(w)$ and $u'\in \mathrm{Min}(w')$ such that $u'\leqdot u$;
\item[\rm (iii)] There is $u'\in \mathrm{Min}(w')$ such that $u' \leqdot w$;
\item[\rm (iv)] For all $u \in \mathrm{Min}(w)$, there is $u' \in \mathrm{Min}(w')$ such that $u' \leqdot u$.
\end{itemize}
\end{theorem}

\begin{proof}
It is clear that (iv) $\Rightarrow$ (iii) $\Rightarrow$ (ii). The
fact that (ii) $\Rightarrow$ (i) follows from
Corollary~\ref{coro:no_matter_which_minimal_above}, noting that
$\ell(w')=\ell(u')=\ell(u)-1=\ell(w)-1$, which forces
$w'\leqdot_{\mathcal{O}} w$. A similar argument shows (iii)
$\Rightarrow$ (iv).

It remains to show that (i) $\Rightarrow$ (iii). If
$w'\leqdot_\mathcal{O} w$, then by Lemma~\ref{minimal_below} there is $y\in\mathrm{Min}(w')$ such that $y\leq w$. We claim
that $w$ covers $y$ in Bruhat order, so that (iii) holds. Otherwise,
we have $\ell(y)+2\leq \ell(w)$, and Lemma~\ref{equiv_covers} yields
$v\in \cM$, say $v\in\mathrm{Min}(w'')$, such that $y< v<w$. Then by
definition of $\leq_\mathcal{O}$ we have $w''<_\mathcal{O} w$, and
by Corollary~\ref{coro:no_matter_which_minimal_above} we get $w'
<_\mathcal{O} w''$: together it contradicts $w'\leqdot_\mathcal{O}
w$.
\end{proof}

\begin{corollary}
\label{C-graded} The poset $(W(I,J,K), \leq_{\mathcal{O}})$ is
graded by the restriction of the length function to $W(I,J,K)$.
\end{corollary}

\begin{proof}
In view of Theorem~\ref{covers}, we have that if $w, w'\in W(I,J,K)$
satisfy $w'\leqdot_{\mathcal{O}} w$, then $\ell(w')=\ell(w)-1$. This
concludes the proof.
\end{proof}

\section{Inclusions of orbit closures in type $A$}

\label{type_a}

The purpose of this section is to illustrate the combinatorial
results obtained in Section \ref{cox} in a situation that is related
to the framework considered in Parts \ref{part1}--\ref{part2}.

In this section, let $e\in \gl_n(\K)$ be a $2$-nilpotent matrix of
rank $r\leq \frac{n}{2}$, say
\begin{equation}
\label{e} e=\begin{pmatrix} 0 & 0 & 1_r \\ 0 & 0 & 0 \\ 0 & 0 & 0
\end{pmatrix}
\end{equation}
so that the stabilizer $Z:=Z_G(e)$, with $G=\mathrm{GL}_n(\K)$, is
given by
\begin{equation}
\label{Z-typeA} Z=\left\{\begin{pmatrix}
a & * & * \\ 0 & b & * \\
0 & 0 & a
\end{pmatrix}: a\in\mathrm{GL}_r(\K),\ b\in\mathrm{GL}_{n-2r}(\K)\right\}.
\end{equation}
In the notation of Section~\ref{cox}, let
$W(I,J,K)\subset\mathfrak{S}_n$, with $I=\{s_1, \dots, s_{r-1}\}$,
$J=\{s_{n-r+1}, \dots, s_{n-1}\}$, $K=\{s_{r+1}, \dots,
s_{n-r-1}\}$. The isomorphism $W_I\to W_J$, $x\mapsto x^*$ is given
by $s_i^*=s_{n-r+i}$. Hence
\begin{eqnarray*}
W(I,J,K) & = & W^{J\cup K} \\  & = &
\{w\in\mathfrak{S}_n:w(r+1)<\ldots<w(n-r),\ w(n-r+1)<\ldots<w(n)\}.
\end{eqnarray*}
We also consider the partial order $\leq_\mathcal{O}$ on $W(I,J,K)$
(see Proposition \ref{P-order_abstract}). Let
$B\subset\mathrm{GL}_n(\K)$ be the Borel subgroup of invertible
upper-triangular matrices.

\begin{theorem}
\label{T9.1} With the above notation, there is a one-to-one
correspondence $w\mapsto\mathbb{O}_w$ (resp.
$w\mapsto\mathcal{O}_w$) between the set $W(I,J,K)$ and the set of
$Z$-orbits on $G/B$ (resp. the set of $B$-orbits on $G\cdot e$).
Moreover,
\[
\mathbb{O}_{w'}\subset\overline{\mathbb{O}_{w}}\quad\mbox{(resp.
$\mathcal{O}_{w'}\subset\overline{\mathcal{O}_{w}}$)}\quad\Leftrightarrow\quad
w'\leq_\mathcal{O}w.
\]
In particular, the cover relations for the inclusions of orbit
closures are described in Theorem \ref{covers}.
\end{theorem}

The orbit $G\cdot e\subset \gl_n(\K)$ is the set of $2$-nilpotent
matrices of rank $r$. The topology of the $B$-orbits on the set of
$2$-nilpotent matrices has been studied in \cite{BP,BR}. In
particular the parametrization of orbits and the characterization of
the inclusion relations between orbit closures given in Theorem
\ref{T9.1} is essentially given in \cite[Lemma 7.3.1]{BP}; however,
we have not understood all the arguments given in \cite{BP}: see the
comment after Lemma \ref{BP}. For the sake of completeness, we give
a proof of Theorem \ref{T9.1}, which is mainly based on \cite{BR}.
We prove Theorem \ref{T9.1} in two steps: in Section
\ref{section-9.1}, we first define the bijective map
$w\mapsto\mathcal{O}_w$. In Section \ref{section-9.2}, we prove the
assertion on inclusion of orbit closures. To do this, we first
recall Boos--Reineke's criterion for inclusion of orbit closures,
and we then show the equivalence between this criterion and the
inequality $w' \leq_{\mathcal{O}} w$ (Proposition \ref{prop_equiv}),
by using the version of Boos--Reineke's criterion stated in Lemma
\ref{comb}. We only prove the result claimed for the $B$-orbits on
$G\cdot e$. The analogous claim for the $Z$-orbits on $G/B$ ensues,
due to the correspondence between these two orbit sets; see also
Remark \ref{Remark-ZvsB}.

\medskip

We point out that the characterization of the cover relations
obtained in Theorem \ref{T9.1} appears to be new. In \cite[Theorem
4.6]{BR}, the authors give a list of elementary relations which
include all the cover relations, but the obtained characterization is only necessary and not sufficient.

\subsection{Parametrization of orbits by oriented link patterns}\label{section-9.1}

In~\cite{BR}, the $B$-orbits on $2$-nilpotent matrices in
$\gl_n(\K)$ are parametrized by oriented link patterns. Let
$\mathcal{D}_{n}$ be the set of \textit{oriented link patterns} on
$\{1,2,\dots,n\}$, that is, oriented graphs on the set $\{1,2,\dots,
n\}$ such that every vertex is incident with at most one arrow. Let
$\mathcal{D}_{n,r}\subset\mathcal{D}_n$ denote the subset of
oriented link patterns with $r$ arrows.

 Hereafter,
$\{\varepsilon_1,\ldots,\varepsilon_n\}$ denotes the standard basis
of $\mathbb{K}^n$. Given $d\in\mathcal{D}_n$, a representative of the corresponding
$B$-orbit $\mathcal{O}_d$ is given by the matrix $M_d\in \gl_n(\K)$
defined by $M_d(\varepsilon_i)=\varepsilon_j$ if there is an arrow
from $i$ to $j$ and $M_d(\varepsilon_i)=0$ if there is no arrow
starting from $i$.

For $k\in\{1,\ldots,n\}$, we write $p_k^d$ for the number of
vertices to the left of $k$ (i.e., $\leq k$) which are not incident
with an arrow, plus the number of arrows whose target vertex lies to
the left of $k$. For $k\in\{0,1,\ldots,n\}$ and
$\ell\in\{1,\ldots,n\}$, we write $q_{k,\ell}^d$ for $p_{\ell}^d$
plus the number of arrows whose source vertex lies to the left of
$\ell$ and target vertex lies to the left of $k$. Note that
$q_{0,\ell}^d=p_\ell^d$. Then Boos and Reineke show:

\begin{theorem}[{\cite[Theorem 4.3]{BR}}]\label{BR}
The set of $2$-nilpotent matrices is the disjoint union of the
orbits $\mathcal{O}_d$ for $d\in\mathcal{D}_n$. Let $d,
d'\in\mathcal{D}_n$, and let us write $d'\leqD d$ if
$\overline{\mathcal{O}_{d'}}\subset \overline{\mathcal{O}_d}$. Then
we have $d'\leqD d$
if and only if
$q_{k,\ell}^d\leq q_{k,\ell}^{d'}$ for all $k\in\{0,1,\ldots,n\}$ and
all $\ell\in\{1,\ldots,n\}$.
\end{theorem}

For $e$ as in (\ref{e}), the orbit $G\cdot e$ is the set of
2-nilpotent matrices of rank $r$, hence we get
\[G\cdot e=\bigsqcup_{d\in\mathcal{D}_{n,r}}\mathcal{O}_d.\]
Then, to obtain the parametrization of the $B$-orbits on $G\cdot e$
claimed in Theorem \ref{T9.1}, we construct an explicit bijection
between the sets $W(I,J,K)$ and $\mathcal{D}_{n,r}$.

The matrix $e$ of (\ref{e}) is of the form $M_{d_1}$, where $d_1$ is
the oriented link pattern with an arrow from $n-r+i$ to $i$ for all
$i\in\{1,\ldots,r\}$. For every $w\in W=\mathfrak{S}_n$, identifying
$w$ with its permutation matrix $w\in\mathrm{GL}_n(\K)$, the matrix
$w\cdot e$ is also of the form $M_{d_w}$ for a unique oriented link
pattern $d_w\in\mathcal{D}_{n,r}$. Specifically, $d_w$ has an arrow
from $w(n-r+i)$ to $w(i)$ for all $i\in\{1,\ldots,r\}$, in other
words, $d_w$ is the oriented link pattern obtained from $d_1$ by
letting $w$ act on the set of vertices $\{1,\ldots,n\}$ in the
canonical way.

Moreover, if $d_w=d_{w'}$, then $w\cdot e=w'\cdot e$, hence
$w'^{-1}w$ centralizes $e$. In view of (\ref{Z-typeA}), the latter
fact implies that the permutation $w'^{-1}w$ is of the form $yxx^*$
with $y\in W_K$ and $x\in W_I$, that is, $w'^{-1}w$ belongs to the
subgroup $W_KW_{I,J}\subset\mathfrak{S}_n$ with the notation of
Section \ref{cox}. Since $W(I,J,K)$ contains exactly one
representative for each $W_KW_{I,J}$-coset, this yields an injective
map
\[
W(I,J,K)\to \mathcal{D}_{n,r},\ w\mapsto d_w.
\]
Finally, since one has $|\mathcal{D}_{n,r}|=\frac{n!}{r!
(n-2r)!}=|W^{J\cup K}|=|W(I,J,K)|$, the above map is bijective. This
yields the bijection $w\mapsto \mathcal{O}_w:=Bw\cdot
e=\mathcal{O}_{d_w}$ of Theorem \ref{T9.1}.

\begin{remark}
\label{Remark-ZvsB} The set of $B$-orbits on $G\cdot e\cong G/Z$ and
the set of $Z$-orbits on $\mathcal{B}=G/B$ are in bijection via the
map $Bg\cdot e\mapsto Zg^{-1}\cdot B$, and this bijection relates
the parametrization of the former set of orbits given above with the
parametrization of the latter set of orbits given in Example
\ref{E7.3}. Specifically, every element $w\in W(I,J,K)$ can be
written as $w=w_1w_2$ for a unique pair $(w_1,w_2)\in W^{I\cup J\cup
K}\times W_I$, and the $B$-orbit $\mathcal{O}_w=Bw\cdot e$
corresponds to the $Z$-orbit $Zw^{-1}\cdot B$ attached to the pair
$(w_1^{-1},w_2^{-1})\in \lW{P}\times \mathfrak{S}_r$, with the
notation of Example \ref{E7.3}.
\end{remark}

\subsection{Inclusion relations between orbit closures}
\label{section-9.2}

Let $0=V_0\subset V_1\subset \ldots\subset V_n=\mathbb{K}^n$ be the
standard complete flag of $\mathbb{K}^n$, so that $B=\{g\in
G:g(V_i)=V_i,\ \forall i\}$. Given a $2$-nilpotent matrix $y\in
\gl_n(\mathbb{K})$ and
$(i,j)\in\{1,\ldots,n\}\times\{0,1,\ldots,n\}$, we write
$$r(i,j,y):=\dim(y(V_i)+V_j).$$ Note that the mapping $y\mapsto
r(i,j,y)$ is constant on every $B$-orbit.

\begin{lemma}
\label{BP} Let $d,d'\in\mathcal{D}_{n,r}$ and write $w,w'$ for the
corresponding elements in $W(I,J,K)$ in the sense of Section
\ref{section-9.1}, i.e., $d=d_w$ and $d'=d_{w'}$. We have $d'\leqD
d$ (that is, $\overline{\mathcal{O}_{w'}}\subset
\overline{\mathcal{O}_{w}}$; see the notation in Theorem \ref{BR})
if and only if $r(i,j,w'\cdot e)\leq r(i,j,w\cdot e)$ for all $i,j$.
In particular, the orbit closure $\overline{\mathcal{O}_w}$, where
$w\in W(I,J,K)$, is given by
$$\overline{\mathcal{O}_w}=\big\{y\in G\cdot e:r(i,j,y)\leq r(i,j,w\cdot e)\ \ \forall (i,j)\in \{1,\ldots,n\}\times\{0,1,\ldots,n\}\big\}.$$
\end{lemma}

The description of $\overline{\mathcal{O}_w}$ given in Lemma
\ref{BP} already appears in \cite[Lemma 7.3.5]{BP}. However, while we
understand the description of the $B$-orbit given in the proof of
\cite[Lemma 7.3.5]{BP}, we do not see how to deduce the given
description of the closure. This proof also contains a reference to
Rothbach's thesis \cite{ROTH}, but for the sake of completeness we
prove Lemma \ref{BP} by using Theorem~\ref{BR}:

\begin{proof}
Let $y=w\cdot e$ and $y'=w'\cdot e$. Assume that $d'\leqD d$.
By Theorem~\ref{BR} we have $q_{k,\ell}^d\leq q_{k,\ell}^{d'}$ for
all $k,\ell$. But $q_{k,\ell}^d$ is simply $\dim
(V_\ell\cap\ker(y))+\dim (y(V_\ell)\cap V_k)$, hence the inequality
gives
\begin{eqnarray*}
&&\dim (V_\ell\cap\ker(y)) + \dim y(V_\ell) - \dim (V_k+y(V_\ell)) \\
&\leq& \dim (V_\ell\cap\ker(y')) + \dim y'(V_\ell) -
\dim(V_k+y'(V_\ell)).
\end{eqnarray*}
But $\dim y(V_\ell) + \dim (V_\ell \cap \ker(y))=\ell$, hence the
inequality can be rewritten as $$\dim(V_k+y'(V_{\ell}))\leq
\dim(V_k+y(V_{\ell})),$$ which implies that $r(\ell,k,y')\leq
r(\ell,k,y)$ for all $k\in\{0,\ldots,n\}$, $\ell\in\{1,\ldots,n\}$.

Conversely, assuming that $r(\ell,k,y')\leq r(\ell,k,y)$ for all
$k\in\{0,\ldots,n\}$, $\ell\in\{1,\ldots,n\}$, we can establish
$q_{k,\ell}^d\leq q_{k,\ell}^{d'}$ by the same inequalities as
above, going the other way around.

Hence we have $d'\leqD d$ if and only if $r(\ell,k,y')\leq r(\ell,k,y)$ for all $k\in\{0,\ldots,n\}$, $\ell\in\{1,\ldots,n\}$. The claimed description of orbit closures ensues.
\end{proof}

We now give an alternative combinatorial criterion to determine
whether an orbit is included in the closure of another orbit.

\begin{notation}
Let $e$ be as in (\ref{e}) above. Let $w\in \mathfrak{S}_n$ and let
$e_w=w\cdot e=wew^{-1}$. Associate a sequence $S_w$ of integers to
$w\in \mathfrak{S}_n$ as follows: the $i$-th number in the sequence
is $j$ if $e_w(\varepsilon_i)=\varepsilon_j$ and $0$ if
$e_w(\varepsilon_i)=0$. Hence, for all $i=n-r+1, \dots, n$, this
sequence has the number $w(i-(n-r))$ in position $w(i)$, and zero
everywhere else. In particular the sequence has $r$ nonzero entries,
all distinct.
\end{notation}

\begin{example}\label{ex:seq}
For $n=4$ and $r=2$, let $w=s_2 s_1 s_3 s_2=(1,3)(2,4)$. Then
$S_w=(3\,4\,0\,0)$.
\end{example}

Note that if $S_w$ has $k$ as nonzero entry, then $S_w$ has $0$ as
$k$-th entry. In the following statement, for every
$i\in\{1,\ldots,n\}$, we denote by $S_w^i$ the truncated sequence
formed by the nonzero entries which are within the first $i$ entries
of $S_w$.

\begin{lemma}\label{comb}
Let $w,w'\in W(I,J,K)$. and let $d_w, d_{w'}$ be the corresponding
oriented link patterns.
We have $d_{w'}\leqD d_w$ if and only if 
$$|\{k\in S_{w'}^i:k> j\}| \leq |\{k\in S_{w}^i:k> j\}|
\quad\mbox{for all $i\in\{1,\ldots,n\}$, all $j\geq 0$}.$$
\end{lemma}

\begin{proof}
This follows immediately from Lemma~\ref{BP} and from the fact that 
\begin{eqnarray*}
r(i,j,w\cdot e) & = & \dim e_w(V_i)+\dim V_j-\dim e(V_i)\cap V_j \\
 & = & |S_w^i|+j- |\{k\in S_w^i:1\leq k\leq j\}|.
\end{eqnarray*}

\end{proof}

\begin{example}
Let $n=4$ and $r=2$, so that $I=\{s_1\}$, $J=\{s_3\}$,
$K=\emptyset$.

(a) Let $w=s_2 s_1 s_3 s_2=(1,3)(2,4)$ be as in Example~\ref{ex:seq}
and let $w'=s_2 s_1=(1,3,2)$. Both elements lie in $W(I,J,K)$. We
have $S_w=(3\,4\,0\,0)$ and $S_{w'}=(0\,3\,0\,1)$, hence
$$S_w^1=(3),\  S_w^2=(3\,4)=S_w^3=S_w^4 \quad\mbox{and}\quad
S_{w'}^1=\emptyset,\ S_{w'}^2=(3)=S_{w'}^3,\ S_{w'}^4=(3\,1).$$
Hence $d_{w'}\leqD d_w$, that is,
$\overline{\mathcal{O}_{w'}}\subset \overline{\mathcal{O}_{w}}$.

(b) 
The sequences $S_w$ associated to the various elements $w\in
W(I,J,K)$ and the corresponding inclusion relations between orbit
closures are described in Figure~\ref{fig:1}.
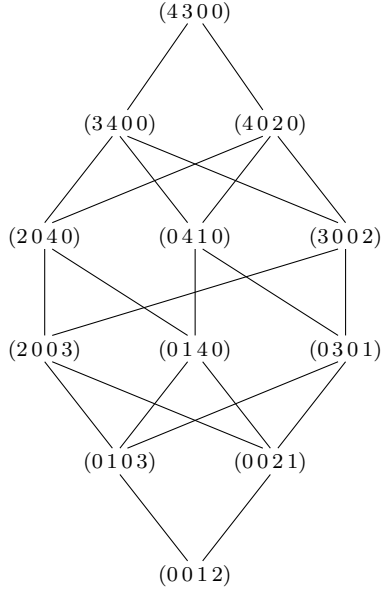
\begin{figure}[htbp]
\begin{center}
\begin{tikzpicture}
\draw (1.9,0.2) -- (1,1.35); \draw (2.1,0.2) -- (3,1.35); \draw
(0.9,1.7) -- (0,2.85); \draw (0,3.2) -- (0,4.35); \draw (0,4.7) --
(0.9,5.85); \draw (1.1,6.2) -- (1.9,7.35); \draw (3.1,1.7) --
(4,2.85); \draw (4,3.2) -- (4,4.35); \draw (4,4.7) -- (3.1,5.85);
\draw (2.9,6.2) -- (2.1,7.35); \draw (1, 1.7) -- (1.9,2.85); \draw
(1.1,1.7) -- (3.9,2.85); \draw (3, 1.7) -- (2.1,2.85); \draw
(2.9,1.7) -- (0.1,2.85); \draw (0.1,3.2) -- (3.9,4.35); \draw
(1.9,3.2) -- (0.1,4.35); \draw (2,3.2) -- (2,4.35); \draw (3.9,3.2)
-- (2.1,4.35); \draw (0.1,4.7) -- (2.9,5.85); \draw (1.9,4.7) --
(1,5.85); \draw (3.9,4.7) -- (1.1,5.85); \draw (2.1,4.7) --
(3,5.85);
\draw (2,0) node{\footnotesize $(0\,0\,1\,2)$}; \draw (1,1.5)
node{\footnotesize $(0\,1\,0\,3)$}; \draw (3,1.5) node{\footnotesize
$(0\,0\,2\,1)$}; \draw (0,3) node{\footnotesize $(2\,0\,0\,3)$};
\draw (2,3) node{\footnotesize $(0\,1\,4\,0)$}; \draw (4,3)
node{\footnotesize $(0\,3\,0\,1)$}; \draw (0,4.5) node{\footnotesize
$(2\,0\,4\,0)$}; \draw (2,4.5) node{\footnotesize $(0\,4\,1\,0)$};
\draw (4,4.5) node{\footnotesize $(3\,0\,0\,2)$}; \draw (1,6)
node{\footnotesize $(3\,4\,0\,0)$}; \draw (3,6) node{\footnotesize
$(4\,0\,2\,0)$}; \draw (2,7.5) node {\footnotesize $(4\,3\,0\,0)$};
\end{tikzpicture}
\end{center}
\caption{The sequences $S_w$ and the inclusion relations between
orbit closures in type $A_3$ for $r=2$. Here $I=\{s_1\}$,
$J=\{s_3\}$, $K=\emptyset$.} \label{fig:1}
\end{figure}
\end{example}

We aim to prove the claim on inclusion of orbit closures in Theorem
\ref{T9.1} by using the criterion from Lemma \ref{comb}. We will use
extensively the \textit{tableau criterion} for the strong Bruhat
order on the symmetric group (see for instance~\cite{BjBr}): if
$x\in \mathfrak{S}_n$, we write $x=x_1\, x_2\,\cdots\, x_n$ if
$x(i)=x_i$. This is called the \textit{line notation} of $x$. Then
given $x,y\in \mathfrak{S}_n$, we have $x\leq y$ (here $\leq$
denotes the strong Bruhat order) if and only if whenever $1\leq k
\leq i$, we have $x_{k,i}\leq y_{k,i}$, where $x_{k,i}$ is the
$k$-th entry of the sequence obtained from $x_1 x_2\cdots x_i$ by
reordering the entries increasingly.

Let $w, w'\in W(I,J,K)$. We consider the partial order
$\leq_\mathcal{O}$ of Proposition \ref{P-order_abstract}. Recall
that, by definition, the relation $w' \leq_{\mathcal{O}} w$ holds if there is $u\in
[w']=w' W_K W_{I,J}$ such that $u \leq w$. Note that this is not
equivalent to having $w'\leq w$ as there might be several elements
of minimal length in the coset $[w']$ (see Section~\ref{cox}).

The proof of Theorem \ref{T9.1} is complete once we show:

\begin{proposition}\label{prop_equiv}
We have $w' \leq_{\mathcal{O}} w$ if and only if $d_{w'} \leqD d_w$.
\end{proposition}

\begin{proof}
Assume that $w'\leq_{\mathcal{O}} w$. For showing that $d_{w'}\leqD
d_w$, fix $i\in\{1,\ldots,n\}$ and $j\geq 0$, and let us check the
condition of Lemma \ref{comb} on the sequences $S_w^i$ and
$S_{w'}^i$.

\begin{thomas}
La notation $w_1, w_2$ peut preter à confusion avec la Section $8$. Mais la aussi, je ne sais pas s'il y a un reel risque.
\end{thomas}

Since $w=w_1\,w_2\,\cdots\,w_n\in W(I,J,K)=W^{J\cup K}$, we have
that the last $r$ entries in the sequence are increasing. Note that
they give the positions of the nonzero entries in $S_w$, and that
these entries are $w_1, w_2,\dots, w_r$. Since
$w_{n-r+1}<w_{n-r+2}<\cdots<w_n$,
the entries $w_k$ with $k\geq n-r+1$ and $w_k\leq i$ form a
subsequence of the form
$w_{n-r+1}\,w_{n-r+2}\,\cdots\,w_{n-r+\ell}$ for some $\ell$. This implies that the
sequence $S_w^i$ is
\[S_w^i=(w_1\,w_2\,\cdots\,w_\ell).\]

The relation $w'\leq_{\mathcal{O}} w$ means that there is $u\in[w']$
such that $u\leq w$. Since $u$ belongs to the coset
$[w']=w'W_KW_{I,J}$, there is $z\in\mathfrak{S}_r$ such that
$$u(n-r+k)=w'(n-r+z_k)\quad\mbox{and}\quad u(k)=w'(z_k)\ \mbox{ for all $k\in\{1,\ldots,r\}$},$$
and this implies that $S_u=S_{w'}$, hence $S_u^i=S_{w'}^i$.

Since $u\leq w$, we must have $u_m> i$ whenever $m> n-r+\ell$,
otherwise we would have a contradiction with the tableau criterion.
(Indeed, the integers $w_{k,n-r+\ell}$ for $1 \leq k \leq n-r+\ell$
contain all the integers from $1$ to $i$, whereas for the integers
$u_{k,n-r+\ell}$, the integer $u_m$ would be missing.) It follows
that the entries of the sequence $S_u^i$ are
among $u_1,u_2,\ldots,u_{\ell}$.
Since $u\leq w$, it follows from the tableau criterion that
$$|\{ k\in\{u_1, \ldots, u_\ell\}: k> j \}| \leq |\{ k\in\{w_1, \ldots, w_\ell\}: k> j \}|.$$
Since $S_w^i=(w_1\,w_2\,\cdots\,w_\ell)$ and $S_u^i$ has entries
among $u_1,u_2,\ldots,u_\ell$, we get that
$$|\{k\in S_{w'}^i:k> j\}|=|\{k\in S_u^i:k> j\}| \leq |\{k\in S_{w}^i:k> j\}|.$$
The proof of the first implication is complete.

\medskip

To show the converse, we use \cite[Theorem 4.6]{BR}, where the
covering relations in the poset $(\mathcal{D}_n, \leqD)$ are
described. We may and will assume that there is a covering relation
between $d_w$ and $d_{w'}$, and we have to prove that $w' \leq_\cO
w$. According to \cite[Theorem 4.6]{BR}, we have three types of
relations to consider.

For the first type, we assume that for some $1\leq a<b\leq n$, there
is an arrow $a \to b$ in $d_w$ whereas there is an arrow $b \to a$
in $d_{w'}$. We also assume that the other arrows are the same for
$d_w$ and for $d_{w'}$. From the definition of $d_w$, it follows
that there is an integer $i$ such that $1 \leq i \leq r$,
$w_{n-r+i}=a$ and $w_i=b$. Moreover, setting $u=(a,b)\circ w$, we
have $d_u=d_{w'}$, hence $u$ belongs to the coset
$[w']=w'W_KW_{I,J}$ (see Section \ref{section-9.1}). Since $a<b$ and
$w^{-1}(a)>w^{-1}(b)$, we get $u \leq w$, and so $w' \leq_\cO w$.

For the second type, we are given three integers $1\leq a<b<c\leq n$
such that the only difference between $d_w$ and $d_{w'}$ involves
the vertices $a,b,c$ and corresponds to one edge of the second
diagram in \cite[Theorem 4.6]{BR}. In each case, we note that
$d_{w'}$ can be obtained from $d_w$ by switching two vertices
$a'<b'$ among $a,b,c$, hence $d_{w'}=d_{(a',b')\circ w}$, which
yields $u:=(a',b')\circ w\in[w']$. Moreover, in each case, it turns
out that $w^{-1}(a')>w^{-1}(b')$ (where we use that $w^{-1}(i)$ lies
in $\{1,\ldots,r\}$, $\{n-r+1,\ldots,n\}$, or $\{r+1,\ldots,n-r\}$
depending on whether $i$ is end point of an arrow, starting point of
an arrow, or not incident with any arrow in $d_w$). This implies that
$u\leq w$, and therefore $w'\leq_\mathcal{O}w$.

For the third type, there are $1\leq a<b<c<d\leq n$ such that the
only difference between $d_w$ and $d_{w'}$ involves the vertices
$a,b,c,d$, and corresponds to one of the edges of the third diagram
in \cite[Theorem 4.6]{BR}. Specifically, two cases may arise:
\begin{itemize}
    \item[\rm (a)] $d_{w'}$ is obtained from $d_w$ by switching two vertices $a'<b'$ (among $a,b,c,d$) such that $a'$ is the starting point of an arrow in $d_w$ and $b'$ is the end point of an arrow in $d_w$. In that case, we get $d_{w'}=d_{(a',b')\circ w}$ hence $u:=(a',b')\circ w$ belongs to the coset $[w']$, and we have $u\leq w$, because $w^{-1}(a')>n-r\geq r\geq w^{-1}(b')$.
    \item[\rm (b)] $d_w$ has two arrows $a'\to a''$ and $b'\to b''$ with $a'<b'$, $b''<a''$,
    and $d_{w'}$ is obtained from $d_w$ by switching the two starting points $a',b'$; or, equivalently, by switching the two end points $a'',b''$. Then, the elements $u':=(a',b')\circ w$ and $u'':=(a'',b'')\circ w$ both belong to the coset $[w']$.
    Moreover, it follows from the definition of $d_w$ that $w^{-1}(a')=n-r+w^{-1}(a'')$ and
    $w^{-1}(b')=n-r+w^{-1}(b'')$. Hence, we have either $w^{-1}(a')>w^{-1}(b')$ or $w^{-1}(b'')>w^{-1}(a'')$, and thereby $u'\leq w$ or $u''\leq w$.
\end{itemize}
In each case, the coset $[w']$ contains an element $u$ such that
$u\leq w$. Therefore, $w'\leq_\mathcal{O} w$. Note also that the
third diagram in \cite[Theorem 4.6]{BR} is the same as our diagram
in Figure \ref{fig:1} while the covering relations within the
corresponding poset $(W(I,J,K),\leq_\mathcal{O})$ are given in the
diagram of Figure \ref{fig:avant}.
\end{proof}

\newpage

\addtocontents{toc}{\protect\setcounter{tocdepth}{0}}

\part*{References}

\renewcommand\refname{}

\end{document}